\documentclass[reqno]{article}
\usepackage{mathtools}
\usepackage{amsfonts}
\usepackage{amsmath}
\usepackage{amsthm}
\usepackage{amssymb}
\usepackage{mathrsfs}
\usepackage{graphicx}
\usepackage{caption}
\usepackage{dsfont}
\usepackage{esint}
\usepackage{xcolor}
\usepackage{cases}
\usepackage{subcaption}
\usepackage{mathabx}
\usepackage{geometry}
\usepackage{booktabs,tabularx}
\usepackage[backend=biber,style=numeric,giveninits=true]{biblatex}

\PassOptionsToPackage{obeyFinal}{todonotes}
\usepackage[backgroundcolor=green!40, bordercolor = white, linecolor=green!40, colorinlistoftodos, textsize=footnotesize]{todonotes}
\usepackage[latin1]{inputenc}
\usepackage[sc]{mathpazo}
\usepackage[colorlinks]{hyperref}

\usepackage{amssymb,amsmath,amsthm,amsfonts, mathrsfs, mathtools}

\theoremstyle{plain}
\newtheorem{theorem}{Theorem}[section]

\newtheorem{corollary}[theorem]{Corollary}
\newtheorem{lemma}[theorem]{Lemma}

\newtheorem{proposition}[theorem]{Proposition}

\theoremstyle{definition}

\DeclareMathOperator{\arcsinh}{arcsinh}

\newcommand{\R}{\mathbb{R}}

\newcommand{\ob}{\overline{B_\rho}}

\newcommand{\clo}{\overline{\Omega}}
\newcommand{\cuo}{{\mathrm C}^1(\overline{\Omega})}

\newcommand\qtext[1]{\quad \text{#1} \quad}

\title{The Influence of Exclusion Zones on the Coexistence of Predator and Prey with an Allee Effect}

\author{
Henri Berestycki\\
\small Department of Mathematics, University of Maryland, College Park, MD 20742, USA\\
\small \'Ecole des Hautes \'Etudes en Sciences Sociales, Paris, France\\
\small Institute for Advanced Study, Hong Kong University of Science and Technology, Hong Kong\\
\small \texttt{hb@ehess.fr}
\and
William F. Fagan\\
\small Department of Biology, University of Maryland, College Park, MD 20742, USA\\
\small \texttt{bfagan@umd.edu}
\and
Alex Safsten\\
\small Department of Mathematics, University of Maryland, College Park, MD 20742, USA\\
\small \texttt{safsten@umd.edu}
}

\date{\today}

\begin{document}

\maketitle

\begin{abstract}
We propose a reaction--diffusion model of predator--prey interaction in which the predators occupy only a subset of the prey's territory, leaving a predator-free \emph{exclusion zone}. Ecological examples include marine protected areas where it is illegal to fish, or buffer zones left between the territories of rival predators. The prey are subject to a strong Allee effect, so excessive predation may lead to the extinction of both species. The exclusion zone mitigates this problem by providing the prey with a refuge in which to proliferate without predation. Thus, paradoxically, a smaller predator territory may be able to support a more substantial population than a larger one. Using a topological degree argument, we show in any dimension that, provided the exclusion zone is large enough, the system possesses spatially heterogeneous \emph{coexistence equilibria} with positive populations of both species. This result is global in the sense that it does not rely on local bifurcations from semi-trivial stationary states. We also show that, as the predator domain becomes asymptotically small, the total predator population does not vanish and in fact approaches an explicit positive limit. In some cases, the total predator population may actually be maximized in this limit of shrinking predation area. Conversely, we show that as the predator domain becomes large, it may exhibit thresholding behavior, passing suddenly from a regime with coexistence solutions to one in which extinction becomes unavoidable, highlighting the need for careful analysis in the management of predator--prey systems.
\end{abstract}


\medskip


\section{Introduction}
\subsection{The role of exclusion zones in predator--prey systems}
Spatial aspects of consumer-resource interactions have long been of interest in mathematical biology, both because of the challenges such problems present for analysis and because results can have practical relevance in applied ecology. Especially interesting problems arise when consumers and their resource species employ different movement behaviors and/or use space in different ways. For example, mammalian carnivores may spatially concentrate their activity in home ranges~\cite{moorcroftlewis, noonan} or along preferred travel routes~\cite{catdog}, creating a patchy composite of landscape zones with different densities of consumer and resource species~\cite{lewismurray}. 

A somewhat related phenomenon is to be found in the spatial organization of competing (i.e., mutually hostile) groups of predators. For example, groups of predators such as packs of wolves or coyotes often divide available space into sharply defined and defended regions, leading to territory formation~\cite{moorcroft}. As a consequence of this division, the interfaces between territories constitute buffer zones into which predators avoid venturing and where prey thrive~\cite{lewismurray}. It was shown~\cite{BerestyckiZilio} that such territoriality can be beneficial for the overall predator population in spite of casualties from the intraspecific hostility. Thus, in practice, the buffer zones between territories act as \emph{exclusion zones} in which predators are excluded and prey can proliferate without predation. 

Exclusion zones also arise in a structurally similar problem in applied ecology in the context of ``marine protected areas'' or MPAs~\cite{botsford, edgar, kriegl, nowakowski} in which human fishers (the consumers) are legally restricted from entering certain portions of a marine landscape so as to facilitate recovery of harvested species. The general idea is that when economically important resource species are shielded from harvest in protected (e.g., consumer-free or reduced-consumer) zones, the resource populations in those zones can build up to high levels and, through dispersal, contribute harvestable individuals to areas outside the protected zones~\cite{botsford}. Marine protected areas have been established in commercial fisheries regions worldwide~\cite{edgar, kriegl}, and the envisioned benefits of dispersal-supported harvest have been reported in many instances, both in terms of local augmentation of harvest~\cite{moland, marcos, nowakowski, sullivan-stack} and in terms of colonization of more distant sites~\cite{leport, lima}.

A variety of mathematical models have been developed to investigate the general utility of MPAs and to explore the ecological and harvest consequences of the number, size, stringency, and placement of MPAs~\cite{neubert, walters, DuShi2006}. Two patch models, wherein one patch is home to both consumers and resources, and the other patch only has the resource species, greatly simplify the spatial dynamics of MPAs, but make clear how a spatial exclusion zone can buffer the resource species away from extinction and allow the consumer population to persist~\cite{ghosh, gonzalezolivares}. Spatially explicit models, ranging from models with one-dimensional patch arrays mimicking coastlines~\cite{walters} to two-dimensional simulation models~\cite{fulton}, make clear how the magnitude and rate of dispersal of the resource species out of the exclusion zone can facilitate consumer persistence, determine feasible harvest levels, and influence spatial harvest strategies~\cite{langebrake, pilyugin}.  Models with Allee effects in the population growth dynamics of the resource species illustrate how important the size of the exclusion zone (or network of such zones) is to persistence of the resource species in the exclusion zone (and the system more generally)~\cite{aalto}. In a diffusive consumer-resource model with a consumer exclusion zone, a critical size for the exclusion zone exists that guarantees persistence of the resource species~\cite{DuShi2006}.

Here, we build on these efforts in several ways. We employ a partial differential equations approach to explore the dynamics of diffusive, spatially explicit models, first in one spatial dimension, and subsequently multiple dimensions. A resource species can inhabit the full spatial domain of these models, but the consumer species (e.g., a fishing fleet) is restricted to a subdomain in keeping with the concept of a strictly enforced MPA. We focus our attention on the location of the boundary of the subdomain and investigate how the size of the consumer-occupied subdomain determines the spatial dynamics of the coupled consumer-resource system. We identify opportunities for the existence of different types of equilibria and bifurcations among them, and further, we characterize the spatial profiles of the resource and consumer species within the model's domain.

Of particular interest is the regime in which the consumer's domain becomes asymptotically small. It may seem that in this case, the consumer population must vanish as its harvest territory becomes small. However, we shall see that their population approaches a positive limit, and in some circumstances, the smaller the consumer-occupied subdomain becomes, the more consumers can be supported. This counterintuitive result shows the subtle interplay between the geometry of consumer and resource territories and the biological mechanisms governing population dynamics.

\subsection{The Predator--Prey model with exclusion zones}\label{sec:model}

The most general model we consider is
\begin{equation}\label{eq:higher_dimensional_model}
    \begin{cases}
        u_t-d_u \Delta u=f(u)-\beta \mathbb{1}_{A}uv & x\in \Omega\\
        v_t-d_v \Delta v=-\gamma v+\alpha uv & x\in A\\
        \partial_\nu u(x,t)=0 & x\in\partial\Omega\\
        \partial_\nu v(x,t)=0 & x\in\partial A.
    \end{cases}
\end{equation}
where $\beta$ is the predation rate, $\gamma$ is the mortality rate of predators, $\alpha$ is a conversion rate, $d_u$ and $d_v$ are diffusion coefficients of the prey and predator, respectively, $\Omega\subset\mathbb R^N$ is the prey domain, $A\subset\Omega$ is the predator domain, and $\nu$ represents the outward-pointing normal vector of the boundary of either domain. Moreover, $A\subset\Omega\subset\mathbb R^N$ are open, $\alpha,\beta,\gamma>0$, and $f:[0,1]\to\mathbb R$ is a smooth function with the following properties: 
\begin{enumerate}
    \item[(F1)] $f$ is continuous.
    \item[(F2)] There are exactly three roots of $f$: $0$, $1$, and $\theta\in (0,1)$.
    \item[(F3)] For $x\in (\theta,1)$, $f(x)>0$ and for $x\in(0,\theta)$, $f(x)<0$.
    \item[(F4)] $\int_0^1 f(x)\,dx>0$.
\end{enumerate}
We consider only solutions $u$ and $v$ so that $0\leq u(t,x)\leq 1$ for all $x\in \Omega$ and $0\leq v(t,x)$ for all $x\in A$. Note that these definitions stipulate that $\Omega\setminus A$ is the predator-free exclusion zone (equivalently, the MPA) and that both predators and prey may occur in $A$; this will later prove notationally convenient. Figure~\ref{fig:OmegaA} shows an example of $\Omega$ and $A$.

\begin{figure}[h]
    \centering
    \includegraphics[width=.7\linewidth]{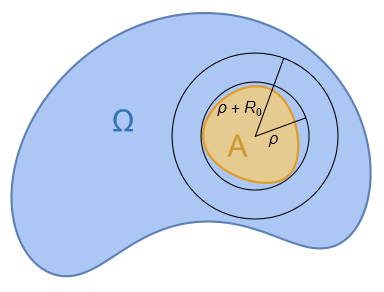} 
    \caption{A diagram of $A\subset\Omega$ illustrating the assumptions of Theorem~\ref{thm:existenceND}.}
    \label{fig:OmegaA}
\end{figure}

As a first result about the model~\eqref{eq:higher_dimensional_model}, we prove that the predator population can only persist if the predators' conversion rate $\alpha$ exceeds their mortality rate:
\begin{proposition}\label{prop:predator_extinction}
    If $\alpha<\gamma$, then all solutions to~\eqref{eq:higher_dimensional_model} with $0\leq u\leq 1$ and $v\geq 0$ satisfy $\lim_{t\to\infty}\Vert v\Vert_{L^1}=0$.
\end{proposition}
\begin{proof}
    Suppose that $(u,v)$ is a solution to~\eqref{eq:higher_dimensional_model} with $0\leq u(t,x)\leq 1$ and let $V(t)=\int_A v(t,x)\,dx$. Then
    \begin{equation*}
        \frac{dV}{dt}=\int_A(\alpha u-\gamma)v\,dx\leq -(\gamma-\alpha)V
    \end{equation*}
    By the Gr\"onwall inequality, $\lim_{t\to\infty}V(t)=0$, so $\Vert v(t,\cdot)\Vert_{L^1}\to 0$. 
\end{proof}
Since the analysis shows that predators inevitably go extinct when $\alpha<\gamma$, this parameter regime is not biologically relevant for sustained predator--prey interactions, and we therefore restrict attention to the case $\gamma\leq\alpha$.

Sections~\ref{sec:coexistence_results_higher_dim}-\ref{sec:thin_limit_1D} focus on analytical results for solutions of this model. Much of this analysis concerns equilibrium solutions to~\eqref{eq:higher_dimensional_model}. Clearly, there are three homogeneous equilibria, namely $v=0$ while $u\in\{0,\theta,1\}$. But we are most interested in inhomogeneous \emph{coexistence} equilibrium solutions to~\eqref{eq:higher_dimensional_model} in which  $0<u<1$ and $v>0$. In Section~\ref{sec:coexistence_results_higher_dim}, we present the main existence results of this paper, showing that such coexistence equilibria exist in both one and higher dimensions. We prove the existence of coexistence equilibria in Section~\ref{sec:coexistence_higher_dim}. The resulting existence theorems include as a condition that the predator-free domain $\Omega\setminus A$ must be sufficiently large. In Section~\ref{sec:persistance_due_to_exclusion}, we demonstrate how coexistence equilibria may cease to exist when this condition does not hold. We perform asymptotic analyses of inhomogeneous equilibria as $A$ becomes small (so-called ``thin-limit'' analysis) in one dimension in Section~\ref{sec:thin_limit_1D}.

We proceed to examine numerical solutions to~\eqref{eq:higher_dimensional_model} in Section~\ref{sec:numerics}. We will examine both time dependent and equilibrium solutions to~\eqref{eq:higher_dimensional_model} and address some questions best suited to a numerical approach:
\begin{itemize}
    \item What is the size (or shape) of the predator domain $A$ that maximizes the predator population?
    \item Under what circumstances may periodic solutions occur?
    \item As we shall see in Section~\ref{sec:persistance_due_to_exclusion}, for certain parameter values, there are no coexistence equilibria when $A=\Omega$, but there are such equilibria when $\Omega\setminus A$ is sufficiently large. How exactly does this transition between existence and non-existence occur? 
\end{itemize}

\subsection{Coexistence solutions in higher dimensions}\label{sec:coexistence_results_higher_dim}

One of the first questions we may ask about the model~\eqref{eq:higher_dimensional_model} is about the existence of coexistence equilibria. We consider the predator--prey model, where predation is limited to a predation zone $A\subset\Omega\subset\mathbb R^N$ (we are most interested in the cases $N=2,3$). We assume $A$ to be a smooth subset (connected or not) of $\Omega$. We assume here that $A$ does not share part of its boundary 
with that of $\Omega$, that is $\overline{A}\subset\Omega$.

We consider the system
\begin{equation}\label{statND}
    \begin{cases}
        - d_u \Delta u = f(u) - \beta \mathbb{1}_A u v, \qtext{in} \Omega,\\
        -d_v \Delta v = \alpha u v - \gamma v, \qtext{in} A, \\
        \partial_\nu u = 0 \qtext{on} \partial\Omega,\\
        \partial_\nu v = 0, \qtext{on} \partial A.
    \end{cases}
\end{equation}
The indicator function of the set $A$, $\mathbb{1}_A$ is defined as usual by  $\mathbb{1}_A(x)=1$ if $x\in A$, else
$\mathbb{1}_A(x)=0$ if $x\in \clo\setminus A$. In the boundary conditions, 
 $\nu$ stands for the outward unit normal vector on $\Omega$ or $A$ respectively. In this system, the set $A$ represents the predation zone, while $B:= \Omega\setminus\overline{A}$ is the exclusion zone. As usual, we note by $B_r(a)$ the open ball of radius $r$ centered at a point $a$ and we write $B_r= B_r(0)$.

We will prove the following result.
\begin{theorem}\label{thm:existenceND}
  Consider the system~\eqref{statND} with $f$ of bistable type satisfying the previous assumptions. Assume that $\gamma < \alpha$. There exists $R_0$ so that for any $A, \Omega$ and such that $A \subset B_\rho \subset B_{\rho+ R_0} \subset  \Omega$, 
there exists a solution $(u,v)$ of system~\eqref{statND} with $v>0$ in $A$ and $u$ is nonconstant. 

\end{theorem}

Theorem~\ref{thm:existenceND} shows that predators and prey may coexist provided that the prey have a large enough predator-free exclusion zone. We will see that the proof of this Theorem yields the additional result that if the gap between the predator domain and the edge of the prey domain is sufficiently large, the prey population density away from the predation zone can be made arbitrarily close to 1. That is, the prey population far from the predator domain is relatively unaffected by the predators. 

The condition of Theorem~\ref{thm:existenceND} that $A$ is a proper subset of $\Omega$ cannot be dropped in general. As we shall see in Section~\ref{sec:persistance_due_to_exclusion}, if $A=\Omega$, then no coexistence equilibrium is possible if $\gamma$ is sufficiently small. In this case, the predators are too effective for their own good. They over-hunt the prey, driving them below their own Allee threshold, leading to extinction of the prey followed by the extinction of the predators.

\subsection{Related literature}

The study of diffusive predator--prey models with the type of spatial heterogeneity introduced via an exclusion zone has developed since the foundational 2006 paper of Du and Shi~\cite{DuShi2006}, which introduced the notion of a \emph{protection zone} as a portion of the prey species' territory from which the predators are excluded. In this work, it is demonstrated that the protection zone fundamentally alters the system dynamics. They use a Holling Type II predation term and logistic growth for the prey. The existence of a critical size for the protection zone, which is expressed through the principal eigenvalue of the Laplacian, determines whether the prey can persist despite predation. When the protection zone exceeds this threshold, the prey persists there, even under high predator growth rates. On the other hand, smaller protection zones may lead to extinction (they consider a generalist predator species which may survive even without the prey species, hence the prey can be hunted to extinction without a resulting extinction of the predators). This framework established the concept of protection zones and emphasized the delicate interplay between spatial geometry, diffusion, and persistence in reaction-diffusion systems.

Subsequent work~\cite{CuiShiWu2014} has extended the concept of protection zones in several directions. In particular, the impact of the strong Allee effect for the prey is explored in~\cite{CuiShiWu2014}. It is shown again that the refuge provided by a protection zone can prevent ``over-exploitation'' wherein excessive predation drives both species to extinction. Specifically, it is shown that coexistence is possible if the protection zone exceeds a critical size. The analysis reveals bistability and threshold phenomena absent from the monostable formulation of~\cite{DuShi2006}. This line of investigation is further explored in~\cite{MinWang2018} by considering ratio-dependent predation that accounts for predator interference and developing a more detailed bifurcation analysis. In~\cite{MinWang2018}, the authors derive explicit analytical criteria for persistence and extinction, prove global existence and positivity of solutions, and demonstrate as in earlier work that spatially non-uniform solutions can bifurcate from semi-trivial (i.e., prey-only) steady states. Together, these studies establish the joint importance of Allee-type bistability and spatial exclusion in promoting coexistence and reveal how the interplay between nonlinear growth and refuge geometry can generate rich spatial dynamics.

Other studies in this area have focused on the role of protection zones and the interaction of protection zones with other ecological effects included in a model. For example, the effect of protection zones in spreading phenomena and invasion~\cite{DuPengSun2019,SunLei2023,SunHan2020}, the effect of more sophisticated growth and predation terms~\cite{LiWuDong2019, Wang2021}, the inclusion of additional species~\cite{LiWuDong2020}, and the inclusion of advective terms~\cite{MaTang2023}.

Existing studies generally fall into two classes: monostable prey models emphasizing global stability and uniqueness~\cite{LiWuDong2019,Wang2021}, and bistable or Allee-type models focusing on critical protection zone size and bifurcation~\cite{CuiShiWu2014,MinWang2018,DuPengSun2019,SunLei2023}.  In the context of coupled predator--prey systems, most analyses have relied on local bifurcation theory (typically perturbations from semi-trivial equilibria) to establish the existence of coexistence equilibria~\cite{CuiShiWu2014,MinWang2018,Yang2023}.  Consequently, the resulting positive equilibria are confined to parameter regimes near the onset of bifurcation, and these approaches offer little information about coexistence far from such thresholds.  Missing from the literature is a discussion of steady-state solutions far from and independent of bifurcation points. Also absent is an analytical treatment of the case of large protection zones when predators are confined to an asymptotically small region and a numerical demonstration of solution behaviors far from equilibrium and bifurcations associated with these behaviors. Finally, the focus of previous work is consistently on the effect of protection zones on the prey population, with little emphasis on the consequences for predators.

In an effort to fill these gaps, the present work shifts the perspective from the prey to the predator. While the spatial refuge is typically viewed as a mechanism for protecting prey populations, we instead investigate how it may \emph{benefit the predators} as well by increasing prey availability. In this spirit, we adopt the term \emph{exclusion zone}, synonymous with the ``protection zone'' used in earlier studies. This new terminology emphasizes that rather than focusing on shielding prey, our focus is on the ecological consequences on the predator population of being excluded from a portion of the prey habitat. A primary motivating question is whether there exists an \emph{optimal predator domain} which maximizes the long-term predator population by being large enough to provide a sufficient prey abundance while not being so large that the prey are over-exploited.

Specifically, our paper addresses the gaps in the literature through three principal advances. First, we establish the existence of coexistence equilibria \emph{without appealing to bifurcation from semi-trivial states}.  Using a novel topological-degree argument, we prove the existence of nontrivial steady states in parameter regimes that may be remote from classical bifurcation points.  This yields a more global and geometrically robust understanding of coexistence, complementing and extending the local bifurcation analyses of, e.g.,~\cite{DuShi2006,CuiShiWu2014,MinWang2018,Yang2023}.  
Second, we develop a rigorous \emph{thin-limit analysis} showing how coexistence persists and transforms as the predator region becomes vanishingly thin.  This asymptotic framework connects the fully coupled PDE model to lower-dimensional effective equations, revealing new behaviors not captured by prior protection-zone studies. Finally, our \emph{numerical investigations} suggest the occurrence of a variety of nonlinear phenomena including multiple coexistence branches, turning-point bifurcations, Hopf bifurcations, and transitions between extinction and persistence regimes. These phenomena highlight the rich dynamical landscape made possible by the coupling of Allee-type prey dynamics and spatial exclusion. Together, these contributions provide the first comprehensive picture of how bistability, geometry, and diffusion interact to shape coexistence in predator--prey systems with exclusion zones.

\section{Coexistence equilibrium in higher dimensions }\label{sec:coexistence_higher_dim}

This section is devoted to the proof of Theorem~\ref{thm:existenceND}. It relies on a fixed point formulation and topological degree argument. 
We denote by $R:= \rho+R_0$ so that $A\subset B_\rho\subset B_R\subset \Omega$. Without loss of generality, we assume that each of these balls is centered at the origin.
We want to show that~\eqref{statND} admits a coexistence solution $(u,v)$ provided $R_0$ exceeds some threshold value.
The proof will be carried out in a sequence of steps. 

\subsection*{Step 1. Auxiliary function in $B_R\setminus \overline{B_\rho}$}
We require some preliminary results on the semi-linear elliptic equation $-d_u \Delta u \,=\, f(u)$ in $\Omega$ and in some subdomains. 
We start by considering the following problem in the annular region $B_R \setminus \overline{B_\rho}$ with $R_0= R-\rho$:
\begin{equation}\label{NzetaBall}
    \begin{cases}
        &-d_u \Delta \zeta_R = f(\zeta_R) \qtext{in} B_R\setminus\overline{B_\rho},\\
       & \zeta_R =0 \qtext{on} \partial B_\rho,\\
        &\partial_\nu \zeta_R = 0 \qtext{on} \partial B_R.
    \end{cases}
\end{equation}
\begin{lemma}\label{lem:zetaR_existence}
 There exists $R_0^*$ such that for any $R_0 \geq R_0^*$, problem~\eqref{NzetaBall}
has a positive solution $\zeta_R$ which is maximum among solutions with values in 
$[0,1]$. Furthermore, $\zeta_R$ is spherically symmetric, 
and increasing with respect to $|x|$.
\end{lemma}
\begin{proof}
    We recall that for a sufficiently large radius $\sigma>0$ , there exists a positive solution of the problem:
    \begin{equation}\label{semilinearBall}
    \begin{cases}
        &-d_u \Delta V \,=\, f(V), \quad V>0, \qtext{in} B_\sigma(0)
        \\
        &V\, = \, 0 \qtext{on}\partial B_\sigma (0),
    \end{cases}
\end{equation}
 (compare e.g.~\cite{Berestyckilions1980}). $V$ is spherically symmetric in $B_\sigma(0)$ and decreasing away from the center. For what follows, we set such a value $\sigma$. For  $R_0^*>2 \sigma$ one can fit a closed ball 
$\overline{B_\sigma(a)}$ of radius $\sigma$, centered at a point
$a\in B_R$ inside the annulus $B_R\setminus\ob$. Then, the function $\underline{u}(x)= V(x-a)$ in 
$B_\sigma(a)$, extended by $\underline{u}(x)= 0$ in the complement of $B_\sigma(a)$ in 
$B_R\setminus\ob$, is a subsolution of problem~\eqref{NzetaBall}. Since $\overline{u}\equiv 1$ is a supersolution, there exists a solution $\zeta_R$. Moreover, we can choose $\zeta_R$ to be the maximum solution of~\eqref{NzetaBall} below 1. Since it can be constructed by evolving the constant supersolution 1, we see that $\zeta_R$ is spherically symmetric.

We write interchangeably $\zeta_R (x) = \zeta_R (r)$ with $r=|x|$. Thus, $\zeta_R$ satisfies
\begin{equation}\label{Nzetaradial}
    \begin{cases}
        &-d_u \Delta \zeta_R= -d_u ( \zeta_R'' + \frac{N-1}{r} \zeta_R' ) = f(\zeta_R) \qtext{in} B_R\setminus\overline{B_\rho},\\
       & \zeta_R(\rho) =0 \qtext{on} \partial B_\rho,\\
       & \partial_\nu \zeta_R = \zeta_R'(R) =0 \qtext{on} \partial B_R.
    \end{cases}
\end{equation}
We know that $\zeta_R\geq V(\cdot -a)$ so that $\zeta_R\not\equiv 0$ and it then follows from the ODE above that $\zeta_R>0$ for $\rho< r \leq R$  and $\zeta_R'(\rho)>0$. We claim that $\zeta_R'(r)>0$ for all
$\rho\leq r < R$. Indeed, suppose not. Then, $\zeta_R'(r_0)=0$ for some $\rho < r_0< R$ and $\zeta_R'(r)>0$ for all $\rho\leq r < r_0$. Then $\zeta_R''(r_0)\leq 0$. If $\zeta_R''(r_0)=0$, then $\zeta_R$ is constant, so $\zeta_R'(\rho)=0$, a contradiction, so $\zeta_R''(r_0)>0$. In this case, $\zeta_R(r)<\zeta_R(r_0)$ for all $r\neq r_0$, for otherwise, for some $r_0<r_1\leq R$ we would have $\zeta_R(r_0)=\zeta_R(r_1)$, but from the equation (multiplied by $\zeta_R'$) we get:
\begin{equation*}
    -\frac{d_u}{2} \zeta_R'(r_1)^2 - d_u (N-1) \int_{r_0}^{r_1} \frac{\zeta_R'^2}{r}dr =0
\end{equation*}
which is impossible.  Then, $\zeta_R(r_0)$ is a global maximum of $\zeta_R$. 

From the ODE satisfied by the solution $V(r)= V(|x|)$ of~\eqref{semilinearBall}, we find that
\begin{equation*}
     F(V(0)):= \int_0^{V(0)} f(s) ds = \frac{d_u}{2} V'(\sigma)^2 + d_u (N-1) \int_0^\sigma \frac{V'(r)^2}{r}dr >0.
\end{equation*}
This entails that $V(0) > \theta'$ where $\theta<\theta'<1$ is defined by  $F(\theta')=0$.

Hence, we then know that $\zeta_R(r_0)> V(0)= \max V > \theta'$. In particular this shows that $f(\zeta_R(r_0))>0$. Thus, the constant $\zeta_R(r_0)$ is a subsolution of the elliptic equation. Setting
\begin{equation*}
    \xi(r) = \begin{cases}
        \zeta_R(r), \qtext{if} \rho\leq r \leq r_0,\\
        \zeta_R(r_0), \qtext{if} r_0\leq r\leq R,
    \end{cases}
\end{equation*}
we have a function of class $C^1$ at the interface $r=r_0$ which is a generalized subsolution (see~\cite{BerestyckiLions1978-subsupersolutions}). Since the supersolution 1 is above it, we know that the maximum solution $\zeta_R$ satisfies $\zeta_R \geq \xi$. Hence, $\zeta_R \geq \zeta_R(r_0)$ for $r_0\leq r \leq R$ which is a contradiction. We conclude that $\zeta_R'(r)>0$
for all $\rho\leq r < R$. 

Note that this proof also yields that $\max \zeta_R > \theta'$.
\end{proof}


\begin{lemma}\label{'lem:z_R_increasing}
The maximum $z_R$ of problem~\eqref{NzetaBall} satisfies:
\begin{itemize}
    \item[(i)] $R\mapsto z_R$ is increasing,
    \item[(ii)] $\lim_{R\to\infty} z_R(R) = 1,$
\end{itemize}
\end{lemma}
\begin{proof}
Let $\rho+R_0^* \leq R < R'$ and let $\zeta_R, \zeta_{R'}$ the corresponding maximum solutions. As in the proof above, we define
\begin{equation*}
    \xi(r) = \begin{cases}
        \zeta_R(r), \qtext{if} \rho\leq r \leq R,\\
        \zeta_R(R), \qtext{if} R\leq r\leq R'.
    \end{cases}
\end{equation*}
This is a subsolution in $B_{R'}\setminus \ob$ and therefore, the maximum solution $\zeta_{R'}$ is above it. Hence, $\zeta_R < \zeta_{R'}$ in $B_R$ which proves (i). 

To prove (ii), we observe that $\zeta_R(x)\nearrow z(x)$ as $R\nearrow\infty$, for all $x, |x|\geq \rho$. We know that $z$ is a solution of
\begin{equation*}
\begin{cases}
    -d_u \Delta z = f(z), \quad z>0 \qtext{in} \R^N\setminus\ob,\\
    z=0 \qtext{on} \partial B_{\rho}.
\end{cases}
\end{equation*}
Furthermore, $z(x)=z(|x|)$ is spherically symmetric, $z(\rho)=0$ and $z>0, z'\geq 0$ for all $\rho <r$. Clearly we must have $f(z(\infty))=0$, and since we know that 
$z(\infty)> \theta'$, we infer that $z(\infty)=1$. This yields the limit in (ii).
\end{proof}

As a consequence of this lemma, we infer the following.
\begin{corollary}\label{ZetaReta}
   For any $0<\eta< 1-\theta$, there exists $R_0=R_0(\eta)\geq R_0^*$ such that for any $R \geq \rho+R_0$, we have
$\zeta_R > 1- \eta$ on $\partial B_R$. 
\end{corollary}


\subsection*{Step 2. Auxiliary function in $\Omega\setminus \overline{B_\rho}$}
We now derive the existence of a solution $\zeta$ in the domain $\Omega\setminus\ob$   of the following problem:
\begin{equation}\label{NzetaOmega}
    \begin{cases}
        &-d_u \Delta \zeta = f(\zeta) \qtext{in} \Omega\setminus\overline{B_\rho},\\
       & \zeta =0 \qtext{on} \partial B_\rho,\\
        &\partial_\nu \zeta = 0 \qtext{on} \partial\Omega.
    \end{cases}
\end{equation}

\begin{lemma}\label{LemmaNzeta}  
    For any $0<\eta< 1- \theta$, let $R_0= R_0(\eta)$ be given by Corollary~\ref{ZetaReta}. Under the assumption that $B_\rho \subset B_{\rho+R_0} \subset \Omega$, there exists a maximum positive solution $\zeta$ of problem~\eqref{NzetaOmega} that satisfies $\zeta \geq \zeta_R$ in $B_R$ and 
    $\zeta \geq 1-\eta$ in $\Omega\setminus B_R$.
\end{lemma}
\begin{proof}
    We make use of the same type of construction of a generalized subsolution as before. We set:
    \begin{equation*}
    \xi(x) = \begin{cases}
        \zeta_R(r), \qtext{if} \rho\leq r=|x| \leq R,\\
        \zeta_R(R), \qtext{if} x\in\Omega\setminus B_R.
    \end{cases}
\end{equation*}
Note that $\xi$ is of class $C^1$ (but not $C^2)$.  Since $\zeta_R(R)> \theta'>\theta$, the constant $\zeta_R(R)$ is a subsolution. Therefore, the function $\xi$ is a subsolution. Since 1 is a supersolution and $\xi<1$, we derive the existence of $\zeta>0$ solving problem~\eqref{NzetaOmega}. By construction we have
$\zeta(x) > \zeta_R(R)> 1-\eta$ in $\Omega\setminus B_R$.
\end{proof}

\subsection*{Step 3. Solutions in all of $\Omega$}
Let  us now consider the Neumann problem for the semilinear elliptic equation in all of $\Omega$:
\begin{equation}\label{NeumannOmega}
    \begin{cases}
        &-d_u \Delta u = f(u) \qtext{in} \Omega,\\
        &\partial_\nu u = 0 \qtext{on} \partial\Omega.
    \end{cases}
\end{equation}
The next lemma shows that only the solution  $u\equiv 1$ of~\eqref{NeumannOmega} lies above the auxiliary function $\zeta$. It plays a key role in the proof of Theorem~\ref{thm:existenceND}.

\begin{lemma}\label{lem:u=1}
    Let $\eta>0$ be such that $V(0)= \max V < 1-\eta$ and let $R_0= R_0(\eta)$ be given by Corollary~\ref{ZetaReta}, where $V$ is the solution of equation~\eqref{semilinearBall}. Suppose that $B_\rho \subset B_{\rho+R_0} \subset \Omega$, and let $\zeta$ be the auxiliary function given by Lemma~\ref{LemmaNzeta}. Let $u$ be a solution of~\eqref{NeumannOmega} in all of $\Omega$, such that $u\geq \zeta$ in $\Omega\setminus\ob$. Then, 
    $u\equiv 1$.
\end{lemma}
\begin{proof}
The maximum principle shows that $u>0$ in $\overline{\Omega}$. 
 We know that $\zeta_R \geq 
 V(\cdot -a)$, for some $a\in B_R$ such that $B_\sigma(a)\subset B_R\setminus\ob$. Let $m\in (\theta', 1)$ denote the maximum value of $V$: 
 $m= V(0)= \max V$.

 We now apply the {\em sliding method} as in~\cite{BerestyckiNirenberg1991} to show that as we
 shift the subsolution $V(\cdot -a)$ by moving around $a$, it keeps remaining below $u$ in $B_R$. 
 More precisely, we claim that $u(b) > m$ for an arbitrary point $b\in B_R$. To this end, we extend $V$ by $0$ outside 
 $B_\sigma(0)$, let $a_s= (1-s)a + sb$ and set:
 \[
 s^*= \max\; \{ \; \overline{s}\in [0,1]\,; \, u(x) \geq
V(x - a_s) \qtext{for all} x\in B_R, \;
s\in[0, \overline{t}]\; \}.
\]
We know that $t^*\geq 0$. Note that the segment from $a$ to $b$ lies in the ball $B_R$ but may cross the ball $B_\rho$, and the sliding ball $B_\sigma(a_s)$ may have portions outside $B_R$.
Nonetheless, we claim that $s^*=1$.
For if not, for some $s^*\in[0,1)$,  we would have: 
\begin{equation*}
   u(x_0) - V(x_0-a_{s^*})=  \min \, \{ \, u(x) - V(x-a_{s^*}) \,; \, x\in B_R \,\}\, =\, 0,
\end{equation*}
for some $x_0\in \overline{B_R}$. 
On $\partial B_R$, by construction
we have $u \geq 1-\eta > \max V$. We also know that $u(x_0)>0$.
Therefore, $x_0$ must be in $B_R \cap B_{\sigma} (a_{s^*})$ and it is thus an interior minimum.  
Since both $u$ and $V$ satisfy the same semi-linear elliptic equation, the strong maximum principle (or positivity property) shows that $u\equiv V$ which is impossible.

Hence, it must be the case that $s^*=1$, whence it follows that $u(b)\geq m= \max V$. (Actually, the proof shows that $u(b)> m$.) Since $b$ is arbitrary, we have proved that $\min_{B_R} u \geq m$. But we also know that $u\geq 1-\eta >m$ in 
$\Omega\setminus B_R$. Together these inequalities show that $u\geq m$ in all of $\overline{\Omega}$. This entails that 
$-\Delta u \geq 0$ in $\Omega$. The function $u$ reaches its minimum at some point $x_0\in\overline{\Omega}$. By the strong maximum principle (in the interior or at the boundary) we derive that $u$ is constant and since $u\geq m > \theta$, this can only happen when $u\equiv 1$. The proof of the Lemma is thereby complete.
\end{proof}

\subsection*{Step 4. Formulation as a fixed point equation.} 

We formulate the system as a fixed point equation. For $u\in \cuo$ we define
$\lambda= \lambda[u]$ and $\phi=\phi[u]$ as the principal eigenvalue and eigenfunction of the problem:
\begin{equation}\label{eq:Nv-equ}
		\begin{cases}
			-d_v \Delta \phi - (\alpha u - \gamma)\phi =\lambda[u] \phi,  &\qtext{in} A,\\
			\partial_\nu \phi=0, &\qtext{on} \partial A,
		\end{cases}
	\end{equation}
with the normalization condition 
\begin{equation*}
    \phi>0\qtext{in} A, \qtext{and} \max_{x\in \overline{A}} \phi(x) =1.
\end{equation*}
The second equation of the system~\eqref{statND} writes $\lambda[u]= 0$ with $v= s \phi[u]$ for some $s>0$.

We formulate the problem with unknowns $(u,s)$ in the function space
$
\mathcal{E}:= C^1(\overline{\Omega}) \times \mathbb{R}.
$
Let $K$ be a positive constant chosen appropriately large so that $z\mapsto f(z)+Kz -\beta S z$ is increasing over the interval $z\in [0,1]$ where the choice of the value of  
the constant $S$ will be made precise below.
For $(u,s)\in\mathcal{E}$, we define $w:= \mathcal{F}(u,s)$ as the unique solution of
\begin{equation}\label{eq:Nu-equ}
		\begin{cases}
-d_u \Delta w +K w  = f(u) + Ku - \beta \mathbb{1}_A s \phi[u] u & x\in \Omega,\\
			\partial_\nu w= 0 \qtext{on}
            \partial\Omega.&
		\end{cases}
	\end{equation}
With these mappings, system~\eqref{eq:stat1D} is equivalent to the system with unknowns $(u,s)$:
\begin{equation}\label{FPN}
\begin{cases}
    & u= \mathcal{F}(u,s), \\
    &\lambda[u] = 0.
\end{cases}
\end{equation}
The solution is then given as the pair $(u, v= s\phi[u])$ with $s\geq0$. By defining
\[
\mathcal              {G}(u,s):= (\mathcal{F}(u,s), s -\lambda[u]),
\]
the problem then reads as a fixed point equation for the operator 
$\mathcal{G}$.


We look for positive solutions in the subset
\begin{equation*}
\mathcal{D}= D\times (0, S), \quad \text{with}\quad D:=  \{ u\in \cuo \, ; \;\; \zeta < u < 1 \; \text{ in } \;\clo \}, 
\end{equation*}
which is an open subset of $\mathcal{E}$. 
The function $\zeta$ comes from Lemma~\ref{LemmaNzeta}, and is 
 extended to all of $\clo$ by setting $\zeta (x) = 0$ for all $x\in B_\rho$. The positive constant $S$ will be chosen appropriately large later.

\subsection*{Step 5. Topological degree}
We set
\[
H(u,s):= (u- \mathcal{F}(u,s), \lambda[u]), \quad\text{for}\quad (u,s)\in\mathcal{D}.
\]
We will prove the existence of a solution of system~\eqref{statND}, that is, of system~\eqref{FPN}, in $\mathcal{D}$ by proving that the Leray-Schauder degree of $H$ in $\mathcal{D}$ with target point $(0,0)$ is non-zero.  

First, we note that
$H$ is a compact perturbation of the identity in $\mathcal{E}$, that is, $\mathcal{F}$ is a compact mapping:
$\mathcal{E} \rightarrow \cuo$, and $\lambda$ is continuous and bounded on bounded sets in the space $\cuo$.
We will show that the Leray-Schauder topological degree
$ d(H, \mathcal{D}, (0,0)) $
 is well defined as a consequence of a more general result regarding a one-parameter family of operators. 

Let $\sigma$ be a $C^1$ function defined on $\R^+$ such that $\sigma' < 0$, on $(0, S)$, $\sigma(0)=1$ and $\sigma(s)=0$ for all $s\geq S$ where $S$ enters the definition of $\mathcal{D}$. We consider the following homotopy: 
\begin{equation*}
    H_\tau(u,s) \,:= \, \left(\, u - \mathcal{F} (u,s)\, ,
\,     \lambda[(1-\tau)u + \tau \sigma(s)] \,
    \right), \quad \tau\in[0,1].
\end{equation*}
Thus, we have $H_0= H$. The mapping $H_\tau$ is a compact perturbation of the identity. In fact, more precisely, it can be written in the form $H_\tau= \mathrm{Id}_{|\mathcal{E}}- \mathfrak{F}_\tau$ with $\mathfrak{F}$ a continuous mapping from $[0,1]$ into the space of compact operators on $\overline{D}$.


The following a priori information on possible solutions of $H_\tau(u,s)=(0,0)$ is at the core of the proof.
\begin{proposition}
 \label{prop:homotopy}
For all $\tau\in[0,1]$, there is no solution of 
$H_\tau(u,s)= (0,0)$ on $\partial\mathcal{D}$.
\end{proposition}
\begin{proof}
 

The pair $(u,s)$ is a solution of $H_\tau(u,s)= (0, 0)$ if and only if it satisfies

\begin{equation}
\label{eq:u-equ1}
		\begin{cases}
-d_u \Delta u  = f(u) - \beta \mathbb{1}_A s \phi[u] u  \qtext{in} \Omega,\\
			\partial_\nu u= 0 \qtext{on}
            \partial\Omega
		\end{cases}
\end{equation}
together with
\begin{equation}\label{eq:u-equ2}
  h_\tau(u,s) := \lambda[(1-\tau)u + \tau \sigma(s)]=0.
\end{equation}
\vspace{.05cm}

Suppose there is a pair $(u,s)\in\partial D$ which is a solution of $H_\tau(u,s)= (0,0)$ for some 
$\tau\in[0,1]$, that is, $(u,s)$ is a solution of~\eqref{eq:u-equ1}--\eqref{eq:u-equ2}.
We know that $\zeta\leq u \leq 1$ in $\Omega$ and $0\leq s \leq S$.
We need to rule out the various cases that exhaust the description of the boundary of $\mathcal{D}$.

\vspace{.2cm}
{\bf Case 1.} $u(x_0)=0$ for some $x_0\in\overline{\Omega}$.
In $\Omega$, $u$ solves a linear elliptic equation $-\Delta u + q(x) u =0 $ for some bounded $L^{\infty}$ function $q$. By the Agmon-Douglis-Nirenberg estimate,
$u$ is in the Sobolev space $W^{2,p}(\Omega)$ for all $p>1$. Since $u\geq 0$ in $\Omega$, by the strong maximum principle in $W^{2,p}$ with $p>N$ (which holds regardless of the sign of $q$), we get $u\equiv 0$ which is impossible since $u>0$ outside $B_\rho$ (see Theorem 9.6 of~\cite{GilTru2001}).
By the same argument, using the Hopf lemma at the boundary, $u$ cannot vanish on $\partial\Omega$.
Therefore, $u>0$ on $\overline{\Omega}$.
\vspace{.2cm}

{\bf Case 2.} $u(x_0)=\zeta (x_0)$ for some $x_0\in\overline{\Omega}\setminus \overline{B_\rho}$.  
In the region ${\Omega}\setminus \overline{B_\rho}$
both $u$ and $\zeta$ are solutions of the same elliptic equation
 $   -d_u \Delta u = f(u)$ and satisfy the same Neumann boundary condition on $\partial\Omega$.
 Therefore, 
since $u\geq \zeta$, if $u$ and $\zeta$ were to touch at $x_0$, again by the strong maximum principle, we would have $u\equiv \zeta$, which is impossible, as $\zeta$ vanishes on the boundary of $B_\rho$. Therefore $u> \zeta$ on $\overline{\Omega}$. 

\vspace{.2cm}
{\bf Case 3.} $u(x_0)=1$ for some $x_0\in\overline{\Omega}$. We 
notice that $u$ satisfies a linear equation $-d_u \Delta u = f(u) -q(x) u$ for some $L^\infty$ function $q\geq 0$, while $1$ is a supersolution of the same equation with $1\geq u$. Moreover, by the Agmon-Douglis-Nirenberg estimate, $u$ is of class $W^{2,p}(\Omega)$ for all $p<\infty$, and therefore, if $x_0\in\Omega$, by the strong maximum principle we get $u\equiv 1$ in $\Omega$. Now, $u$ is of class $\mathrm{C}^2$ near the boundary of $\Omega$. Hence, from the strong maximum principle at the boundary we  infer $u\equiv 1$ as well if $x_0\in\partial\Omega$. But for $u\equiv 1$ to solve the equation, we must have $s=0$. Then, we get $h_\tau(u,s): = \lambda[(1-\tau)u + \tau \sigma(s)] = \lambda[1-\tau + \tau \sigma(0)]= \lambda[1]= \gamma-\alpha <0$. Hence, $u\equiv 1$ cannot be a solution and we have shown that $u< 1$ on $\overline{\Omega}$.

\vspace{.2cm}
{\bf Case 4.} $s=0$. In this case, the first equation reduces to the equation~\eqref{NeumannOmega} which holds in all of $\Omega$. Since
$u\geq \zeta$, owing to Lemma~\ref{lem:u=1}, we infer that $u\equiv 1$ which we have just ruled out. This lemma plays an important role here. 

\vspace{.2cm}
{\bf Case 5.} $s=S$. This last case is where the choice of $S$ comes into play.
That it cannot happen 
is a consequence of the following lemma.
\begin{lemma}\label{lem:S}
    Given $0< \gamma_0<\gamma$, there exists $S= S(\gamma_0)>0$ (sufficiently large) such that any solution $(u,s)$ of equation~\eqref{eq:u-equ1}
    with $s\geq S$ and $u\in\overline{D}$ must satisfy $\lambda[u]\geq \gamma_0$.
\end{lemma}

At first glance, the $s$ variable does not appear in the second equation of system~\eqref{FPN}, which can be surprising since this scalar equation corresponds to the additional scalar variable $s$. In fact, the previous lemma highlights the way the $s$ variable comes into play in the second equation of this system.

To prove this Lemma, we first require the following observation as in dimension 1.
\begin{lemma}\label{phiboundbelowN}
    There exists $\eta>0$ such that for all $u$
    with $0\leq u \leq 1$, the eigenfunction defined by~\eqref{eq:Nv-equ} satisfies
    \[
0< \eta \leq \phi[u](x) \leq 1 \qtext{for all} x\in A 
\qtext{and} u\in\overline{D}.
\]
\end{lemma}
 \begin{proof}
We know that $\gamma -\alpha \leq \lambda[u]\leq  \gamma$ for $0\leq u\leq 1$. Therefore, from the eigenfunction equation~\eqref{eq:Nv-equ}, and by the normalization condition
$\max_{x\in \overline{A}} \phi[u](x)=1$,
it follows that $\phi[u]$ is bounded from above and away from 0.
For if not, one could find a sequence of functions 
$u_j$ with $0\leq u_j\leq 1$, and corresponding 
$\lambda_j= \lambda[u_j], \phi_j= \phi[u_j]$ solutions of the eigenvalue problem~\eqref{eq:Nv-equ}, and such that
\( \min_{\overline{A}} \phi_j \searrow 0\)
as $j\nearrow \infty$. We can find subsequences (still denoted with index $j$) such that
\[
\lambda_j\to \lambda,\quad  u_j \rightharpoonup  u \qtext{weakly in} L^p(A), \quad
\phi_j \rightharpoonup \phi \qtext{weakly in} W^{2,p}(A), \qtext{for all}
p>1.
\]
From the classical embedding theorems with $p>N$, we know that $\phi_j \to \phi$ in
$C^1(\overline{A})$. Therefore, $\phi\in W^{2, p}(A)$ satisfies
\[
\Delta\phi = q(x) \phi, \quad \phi\geq 0, \qtext{in} A,
\qtext{and} \partial_\nu\phi=0 \qtext{on} \partial A
\]
where $q= (\alpha u - \gamma - \lambda)$ so that $ q\in L^\infty(A)$. Moreover, there must exist a point $x_0\in\overline{A}$ where $\phi$ vanishes. 
The strong maximum principle in $W^{2,p}$ (with $p>N$)  then shows that
$\phi\equiv 0$, which contradicts  $\max_{\overline{A}} \phi=1$. 
Therefore, there is a $\eta>0$ such that
\[
0< \eta \leq \phi[u](x) \leq 1 \qtext{for all} x\in [0,a] 
\qtext{and} u\in\overline{D}.
\]
\end{proof}

Next, we will use the asymptotic behavior when $\kappa\to\infty$ for the solution of the following equation:
\begin{equation}\label{eq:kappa}
    \begin{cases}
        - \Delta w + \kappa w =0 \qtext{in} A,\\
        w=1 \qtext{on} \partial A.
    \end{cases}
\end{equation}
\begin{lemma}\label{lem:boundarylayer}
    The solution $w= w_\kappa$ of equation~\eqref{eq:kappa} for $\kappa>0$
satisfies $0< w_\kappa < 1$ and $w_\kappa \searrow 0$ as $\kappa\nearrow \infty$, uniformly on compact subsets of $A$.
\end{lemma}
\begin{proof}
    This boundary layer type property is classical. For the convenience of the reader we include the proof.
We start by examining the case of the ball $B_\delta$ of radius $\delta<\rho$ about the origin and consider
\begin{equation}\label{eq:kappaball}
    \begin{cases}
        - \Delta z + \kappa z =0 \qtext{in} B_\delta,\\
        z=1 \qtext{on} \partial B_\delta.
    \end{cases}
\end{equation}
Define:
\[
\overline{z}(x): =e^{\xi(r^2 - \delta^2)}, \qtext{for} r=|x|,
\]
where $\xi$ is a parameter. Thus, $\overline{z}=1$ on $\partial B_\delta$ and
\[
-\Delta \overline{z} + \kappa \overline{z}= 
(-4 \xi^2 r^2 - 2N \xi + \kappa )\,
\overline{z}.
\]
Therefore, by choosing $\xi$ such that
$\kappa > 4 \xi^2\rho^2 + 2N \xi$, we get:
\[
-\Delta \overline{z} + \kappa \overline{z} \geq 0.
\]
Hence, $\overline{z}$ is a supersolution of equation~\eqref{eq:kappaball} and thus 
$z(0)\leq \overline{z}(0)= e^{-\xi \delta^2}$. It follows that, given $\delta, \varepsilon >0$, there is a $\kappa_0$ sufficiently large such that, for any $\kappa \geq \kappa_0$, we have $z(0) \leq \varepsilon$.

Now, considering the set $A$ and the equation~\eqref{eq:kappa}, for any point $x_0\in A$ whose distance to the boundary is larger than $\delta$, we can compare $w$ and $z(\cdot - x_0)$ which implies $w(x_0)< e^{-\xi \delta^2}$. We have thus proved that for any $\delta>0$, the solution $w$ converges uniformly to 0 on the set $\{x\in A \, ; \,\mathrm{dist}(x, \partial A)\geq \delta \}$.
\end{proof}
{\em Proof of Lemma~\ref{lem:S}.}
Let  $\sigma_0>0$ be a constant such that $f(u) \leq \sigma_0 u$ for all $u\in\overline{D}$ so that a solution $(u,s)\in\mathcal{D}$ with $s\geq S$ satisfies
\[
-d_u \Delta u  \leq [\sigma_0 - \beta S\delta]u \qtext{on} \overline{A}.
\]
Hence, for any $\kappa>0$ one may choose $S$ sufficiently large so that
\begin{equation} \label{eq:kappa-u}
-\Delta u + \kappa u \leq 0 \qtext{in} A.
\end{equation}
Comparison with~\eqref{eq:kappa} yields $u\leq w$ in $A$. Therefore, given some $\varepsilon>0$ and $\delta >0$, if $S$ is sufficiently large, we get $u\leq \varepsilon$ in the set 
$A_\delta:=\{x\in A \, ; \,\mathrm{dist}(x, \partial A)\geq \delta \}$ while $u\leq 1$ in the complement. 

Let $\lambda=\lambda[u], \phi=\phi[u]$ be the eigen-pair defined by
equation~\eqref{eq:Nv-equ}. 
It satisfies
\begin{equation}\label{eq:Rayleigh}
    \int_A |\nabla \phi|^2 - \int_A \alpha u \phi^2
    = (\lambda - \gamma ) \int_A \phi^2.
\end{equation}
Splitting the second integral over the set $A_\delta$
and its complement $N_\delta$, which is the $\delta$-neighborhood of the boundary in $A$, and using Lemma~\ref{phiboundbelowN} yield:
\begin{equation}\label{eq:integral comparison}
    \frac{\int_A u \phi^2}{\int_A \phi^2} \leq 
    \frac{1}{\eta} \frac{|N_\delta|}{|A|} +\varepsilon.
\end{equation}
Given $0 < \varepsilon < (\gamma - \gamma_0) /2$, we choose $\delta>0$ small enough so that $\frac{1}{\eta} \frac{|N_\delta|}{|A|} < \varepsilon$.
From this choice of $\varepsilon, \delta>0$,  we get a value of $S$ such that for any $s\geq S$, from~\eqref{eq:Rayleigh} and~\eqref{eq:integral comparison}  we infer $\lambda - \gamma \geq 2 \varepsilon$ which yields 
$\lambda[u] \geq \gamma_0$.
This completes the proof of Lemma~\ref{lem:S} which allows us to rule out case 5.
\end{proof}

\subsection*{Step 6. Computation of the topological degree}

The existence of a solution of~\eqref{statND} with $u\geq\zeta$ on $\Omega\setminus A$ is a consequence of the fact that the topological degree $d(H, \mathcal{D}, (0,0))$ is not zero. We now compute this degree.

\begin{proposition}\label{degree}
       $$d(H, \mathcal{D}, (0,0))= 1.$$  
\end{proposition}
\begin{proof}
Owing to Proposition~\ref{prop:homotopy}, the homotopy invariance of the Leray-Schauder degree implies that the degree 
$d(\, H_\tau, \mathcal{D}, (0,0)\,)$ is independent of $t$ and thus,
\[
d(\, H, \mathcal{D}, (0,0)\,)\, = \, d(\, H_1, \mathcal{D}, (0,0)\, ).     
\]
The second component of $H_1$, $h_1(u,s)$, does not depend on $u$. Indeed, we have
\begin{equation*}
    h_1(u,s)= \lambda[\sigma(s)]=  \gamma - \alpha \sigma(s).
\end{equation*}
Since $\gamma < \alpha$ and $\sigma'<0$ on $(0, S)$, the equation
$h_1(u, s)=0$ has a unique solution $s=s_1 \in (0, S)$.
Hence, by a standard homotopy argument, we can diagonalize the operator $H_1$ in its two components:
\begin{equation}\label{multiplicative}
    d(\, H, \mathcal{D}, (0,0)\,)\, = \, d(\, H_1, \mathcal{D}, (0,0)\, )
    = d(\tilde H(\cdot, s_1), D, 0) \times d(h_1, (0, S), 0).
\end{equation}
We know that $d(h_1, (0, S), 0)=1$. 

To compute the other term in the right-hand side we rely on the following:
\begin{lemma}
    Let $K$ be such that the function $z\mapsto f(z) + K z - \beta S z$ is increasing on $[0,1]$. Then, 
    $\mathcal{F}(\cdot, s_1)$ maps $\overline{D}$ into its interior.
\end{lemma} 
\begin{proof}
Let $\zeta\leq u\leq 1$ in $\Omega$.  
Recall that $\zeta$ is extended by $0$ over $B_\rho$ and 
that $w=\mathcal{F}(u, s_1)$ is defined by
\[
\begin{cases}
    -d_u \Delta w +K w  = f(u) + Ku - \beta \mathbb{1}_{A} s_1 \phi[u] u  \qtext{in} \Omega,\\
			\partial_\nu w=0, \qtext{on} \partial \Omega.	
\end{cases}
\]
By the choice of $K$, we know that for such a $u$ we have
\[
\begin{cases}
f(u) + K u - \beta s_1 \phi[u] u \geq 0 &\qtext{for} x\in \Omega,\\
f(u) + Ku \geq f(\zeta) + K\zeta &\qtext{for} x\in\overline{\Omega}\setminus B_\rho,\\
f(u) + K u - \beta s_1 \phi[u] u \leq f(1) + K = K  &\qtext{for} x\in \Omega,\\
f(u) + K u - \beta s_1 \phi[u] u < K  &\qtext{for} x\in A.
\end{cases}
\]
From the third inequality, we get
\[
\begin{cases}
- d_u \Delta(w-1) + K( w-1) \leq 0 \qtext{in} \Omega,\\
\partial_\nu (w-1) =0 \qtext{on} \partial\Omega,
\end{cases}
\]
whence, by the maximum principle, $u \leq1$ on $\overline{\Omega}$. From the last inequality we infer that $u\not\equiv 1$ and thus $u < 1$ on $\overline{\Omega}$.
Likewise, since
$-d_u \Delta w + K w \geq 0$ on $\overline{\Omega}$ and $-d_u \Delta (w-\zeta) + K(w-\zeta) \geq 0$ on $\Omega\setminus A$, we get $w\geq 0$ on $\overline{\Omega}$ and $w\geq\zeta$ on $\overline{\Omega}\setminus B_\rho$. From the fact that $\zeta$ extended by 0 on $B_\rho$ is a strict sub-solution of the equation, we derive that these inequalities are strict and therefore that $w$ is in the interior of $D$.
\end{proof}

We can conclude the proof of Proposition~\ref{degree}. Since $D$ is convex and $\mathcal{F}(\cdot, s_1)$ maps $D$ into its interior, we know that the degree $d(\mathcal{F}(\cdot, s_1), D, 0)$ is 1. Hence, the multiplicative formula~\eqref{multiplicative} above then yields:
\[
d(\, H, \mathcal{D}, (0,0)\,) \,= \, 1.
\]
which establishes the claim of Proposition~\ref{degree}. \qed

Since the degree is non zero, we conclude that there exists at least one solution $(u, s)$ with $0<s$ and $u\in D$ of the system~\eqref{statND}.
The proof of Theorem~\ref{thm:existenceND} is thereby complete.
\end{proof}

\section{Persistence owing to exclusion zones}\label{sec:persistance_due_to_exclusion}

    Theorem~\ref{thm:existence1D} shows that the 1D stationary model~\eqref{eq:stat1D} has solutions provided $L-a$ is large enough. Extending this, Theorem~\ref{thm:existenceND} shows that the higher dimensional model~\eqref{eq:higher_dimensional_model} has coexistence equilibria provided $\Omega\setminus A$ contains a large enough ball. In other words, if the prey have a large predator-free exclusion zone, then both species can persist. In this section we show that the presence of an exclusion zone is, in some circumstances, necessary for coexistence solutions to exist..

    If we consider the homogeneous problem with no spatial dependence in the case that the predator and prey domains coincide, the model~\eqref{eq:higher_dimensional_model} becomes an ODE. It is easy to see that this ODE has a unique coexistence equilibrium $(\hat u,\hat v)$ given by
    \begin{equation}
        \hat u=\frac{\gamma}{\alpha},\quad \hat v=\frac{1}{\beta}\frac{f(\hat u)}{\hat u}
	\end{equation}
    if and only if $\theta<\gamma/\alpha<1$ because otherwise, either $\hat u\not\in (0,1)$ or $f(\hat u)$ (and hence $\hat v$) is negative. This suggests that for small $\gamma/\alpha$, the equilibrium solutions of~\eqref{eq:higher_dimensional_model} may fail to be coexistence solutions since, for such solutions $(u,v)$, we must have $f(u)<0$. In other words, the overabundance of predators resulting from their high conversion rate $\alpha$ and low mortality rate $\gamma$ forces the prey population beneath the Allee threshold. This fact is proved by the following theorem.
    \begin{theorem}
       Suppose $A=\Omega$, a bounded domain in $\mathbb R^N$ with $C^2$ boundary. There exists  $\gamma_0>0$ (depending on $\alpha,\beta,f$, and $\Omega$) such that if $0<\gamma<\gamma_0$, then~\eqref{eq:higher_dimensional_model} does not have a coexistence equilibrium.
    \end{theorem}
    \begin{proof}
        Fix $\alpha$, $\beta$, $f$, and $\Omega\subset\mathbb R^N$ with $\Omega$ bounded. Suppose, by way of contradiction, that there exists a coexistence equilibrium $(u_n,v_n)$ to~\eqref{eq:higher_dimensional_model} with $A=\Omega$ and $\gamma_n=1/n$ for each $n\in\mathbb N$. Since $(u_n,v_n)$ is a coexistence equilibrium, it satisfies $0<u_n<1$ and $0<v_n$.

        Then for each $n$, we have $-d_u \Delta u_n+\beta u_nv_n=f(u_n)$. We multiply by $u_n$ and integrate by parts to obtain
        \begin{equation}
            \beta\left(\inf_{\Omega} v_n\right)\int_\Omega u_n^2\leq d_u\Vert\nabla u_n\Vert_{L^2}^2+\beta\int_\Omega v_n u_n^2=\int_\Omega u_n f(u_n)\leq Q\int_{\Omega }u_n^2,
        \end{equation}
        where $Q=\sup_{0<s\leq 1}\tfrac{f(s)}{s}$. Therefore, for all $n$,  $\inf_\Omega v_n\leq \frac{Q}{\beta}$.
        
        On the other hand, by Lemma~\ref{phiboundbelowN}, there exists $\eta>0$ so that, regardless of $u_n$, a solution $\phi_n$ to
        \begin{equation*}
            -\Delta\phi_n-(\alpha u_n-\gamma)\phi_n=0,\quad\quad\sup_\Omega\phi_n=1
        \end{equation*}
        satisfies $0<\eta\leq\phi_n(x)\leq 1$ for all $x\in\Omega$. But $\phi_n=\frac{v_n}{\Vert v_n\Vert_{L^\infty}}$ is such a solution, so
        \begin{equation*}
            \eta \Vert v_n\Vert_{L^\infty}\leq \inf_\Omega v_n.
        \end{equation*}
        We conclude that for all $n$, $\Vert v_n\Vert_{L^\infty}\leq \frac{Q}{\eta\beta}$. Thus, $\Vert v_n\Vert_{L^\infty}$ is bounded.

        Now we show that $u_n$ and $v_n$ have subsequences that converge in $C^{0}(\Omega)$. For this, we recall that if $w$ solves
        \begin{equation}
        \begin{cases}
            -\Delta w+w=g & x\in\Omega\\
            \partial_\nu w=0 & x\in\partial\Omega,
        \end{cases}
        \end{equation}
        and if $g\in L^p(\Omega)$, then $w$ is in $W^{2,p}(\Omega)$ and there is a constant $K$ (depending only on $\Omega$) such that $\Vert w\Vert_{W^{2,p}}\leq K\Vert g\Vert_{L^p}$ (see, e.g., Theorem 9.26 in~\cite{Brezis}). For each $n$, we may write the equations for $u_n$ and $v_n$ as
        \begin{align*}
            -&\Delta u_n+u_n=g_{u,n}\\
            -&\Delta v_n+v_n=g_{v,n}
        \end{align*}
        where
        \begin{align*}
            g_{u,n}&=u_n+\frac{f(u_n)-\beta u_nv_n}{d_u}\\
            g_{v,n}&=v_n+\frac{(\alpha u_n-1/n)v_n}{d_v}.
        \end{align*}
        Since $f(u_n)$, $u_n$ and $v_n$ are all uniformly bounded in $L^\infty$, there exists $R>0$ so that $\Vert g_{u,n}\Vert_{L^p}$ and $\Vert g_{v,n}\Vert_{L^p}$ are bounded by $R$ for all $n\in\mathbb N$ and $p\geq 1$. Thus,
        \begin{equation}
            \Vert u_n\Vert_{W^{1,p}}\leq KR\quad\text{and}\quad \Vert v_n\Vert_{W^{1,p}}\leq KR
        \end{equation}
        Choosing $p>N/2$, by the Rellich--Kondrachov and Sobolev--Morrey Embedding Theorems, we may strike out subsequences $u_{n_k}$ and $v_{n_k}$ converging to $u_0$ and $v_0$ in $C^0(\Omega)$.

        Let $\tilde v_n=\frac{v_n}{\Vert v_n\Vert_{C^0}}$. Then $\tilde v_n$ satisfies $-d_v\Delta v_n=(\alpha u_n-1/n)\tilde v_n$ in $\Omega$. Therefore, by the same argument as the above, $\tilde v_n$ is uniformly bounded in $W^{2,p}(\Omega)$ for all $p\geq 1$ and it has a subsequence that converges to $\tilde v_0$ in $C^0(\Omega)$. Without loss of generality, we take this subsequence to be $\tilde v_{n_k}$.
        
        For each $k$, we have
        \begin{equation}
            0=d_v\int_\Omega \Delta \tilde v_{n_k}\,dx=\int_\Omega \left(\alpha u_{n_k}-\frac{1}{n_k}\right)\tilde v_{n_k}\,dx.
        \end{equation}
        Thus,
        \begin{equation}
            \int_\Omega u_0\tilde v_0=\lim_{k\to\infty}\frac{1}{\alpha n_k}\int_\Omega \tilde v_{n_k}=0.
        \end{equation}
        Since Lemma~\eqref{phiboundbelowN} applies to $\tilde v_{n_k}$ for each $k$, it also applies to $\tilde v_0$. But we have $\sup_{\Omega}\tilde v_0=1$, so $\tilde v_0>0$ in $\Omega$. Thus, $u_0\equiv 0$. But $u_{n_k}$ converges to $u_0$ uniformly, meaning that for $k$ sufficiently large, $0<u_{n_k}<\theta$ in $\Omega$. But for such $k$,
        \begin{equation}
            \int_{\Omega} d_u
            \Delta u_{n_k}+f(u_{n_k})-\beta u_{n_k}v_{n_k}<0,
        \end{equation}
        which is a contradiction. Thus, there exists $\gamma_0>0$ so that for $\gamma<\gamma_0$, there are no coexistence equilibria to~\eqref{eq:higher_dimensional_model}.
    \end{proof}
    It is worth noting how this proof breaks down if $A$ is a proper subset of $\Omega$. In the latter case, we may still find limits $u_0$ and $v_0$ such that $\int_A u_0v_0=0$, and conclude that $u_0=0$ on $A$, but we may not conclude that $u_0=0$ on $\Omega$. This is why coexistence equilibria may still be possible via Theorem~\ref{thm:existenceND}.

\section{Thin-limit analysis in 1D}\label{sec:thin_limit_1D}

In this section, we consider the 1D steady state problem in the case that $\Omega=(0,L)$ and $A=(0,a)$ for $0<a<L$. That is, we study solutions to
\begin{equation}\label{eq:stat1D}
    \begin{cases}
        -d_u u''=f(u)-\beta \mathbb{1}_{[0,a]}uv & x\in (0,L)\\
        -d_v v''=-\gamma v+\alpha uv & x\in (0,a)\\
        u'(0)=u'(L)=0 & \\
        v'(0)=v'(a)=0 &
    \end{cases}
\end{equation}
with $v\not\equiv 0$.

We note that Theorem~\ref{thm:existenceND} does not apply to the setting~\eqref{eq:stat1D} because $A$ and $\Omega$ share a common boundary at $x=0$. However, Theorem~\ref{thm:existence1D} below shows that the result of Theorem~\ref{thm:existenceND} may be extended to~\eqref{eq:stat1D}. 

As in the higher-dimensional case, we will employ as a lower bound for $u$ in~\eqref{eq:stat1D} the function $\zeta_{a,L}$ defined on $[a, L]$:
\begin{equation}\label{eq:aux1D}
    \begin{cases}
        -d_u \zeta_{a,L}''=f(\zeta_{a,L}) & x\in (a,L)\\
        \zeta_{a,L}(a)=0,\\
        \zeta_{a,L}'(L)=0,\\
        0\leq \zeta_{a,L}<1
    \end{cases}
\end{equation}
Observe that by Lemmas~\ref{lem:zetaR_existence} and~\ref{'lem:z_R_increasing} and Corollary~\ref{ZetaReta}, the following holds: 
\begin{proposition}\label{prop:subsolzet}
There exists $b_0$ sufficiently large such that if $L-a\geq b_0$, then~\eqref{eq:aux1D} admits a maximum positive solution $\zeta_{a,L}$ such that $0<\zeta_{a,L}<1$ for all $a<x<L$. Moreover, the following are true:
\begin{itemize}
    \item For any $\eta>0$ there exists $b=b(\eta)$ so that if $L-a>b$, then $\zeta_{a,L}(L)>1-\eta$.
    \item $\zeta_{a,L}(x)$ is increasing in $x$.
    \item If $a+b_0\leq L_1\leq L_2$, then $\zeta_{a,L_1}\leq \zeta_{a,L_2}$ on $[a,L_1]$.
    \item If $0\leq a_1\leq a_2\leq L-b_0$, then $\zeta_{a_1,L}\geq\zeta_{a_2,L}$ on $[a_2,L]$.
\end{itemize}
\end{proposition}
The one-dimensional nature of $\zeta_{a,L}$ permits additional analysis. For example, multiplying~\eqref{eq:aux1D} by $\zeta_{a,L}'$ and integrating, we find a quantity that is conserved for all $x$ for a single solution $\zeta_{a,L}$:
\begin{equation}\label{eq:energy}
    E(a,L)=\frac{d_u}{2}(\zeta_{a,L}')^2+F(\zeta_{a,L}),
\end{equation}
where $F$ is the antiderivative of $f$:
\begin{equation}
    F(s)=\int_0^s f(s')\,ds'.
\end{equation}
Several properties of $F$ are important: $F(0)=0$, and due to property F4, there is an additional root $\theta'\in (\theta,1)$. For $s\in (0,\theta')$, $F(s)<0$, reaching a minimum at $s=\theta$. For $s\in(\theta',1)$, $F(s)>0$, reaching a maximum at $s=1$. Thus, for all $y\in (0,F(1)]$, there is a unique solution to $F(s)=y$, and that solution is in $(\theta',1)$.

The following analogue of Theorem~\ref{thm:existenceND} guarantees existence of nontrivial solutions to~\eqref{eq:stat1D}.
\begin{theorem}\label{thm:existence1D}
  Consider the system~\eqref{eq:stat1D} with $f$ of bistable type satisfying the assumptions F1-F4. Assume that $\gamma < \alpha$. There exists $b_0$ (at least as large as $b_0$ in Proposition~\ref{prop:subsolzet}) such that for any $a, L$ such that $L-a\geq b_0$, there exists a solution $(u,v)$ of system~\eqref{eq:stat1D} such that $\zeta_{a,L}<u<1$ in $(a,L)$ and $v>0$ in $(0,a)$. In particular, $u$ is nonconstant.
\end{theorem}
\begin{proof}
    To prove this fact, we may simply consider $\tilde\Omega=(-L,L)$ and $\tilde A=(-a,a)$. Then we may repeat the proof of Theorem~\ref{thm:existenceND} for prey domain $\tilde\Omega$ and predator domain $\tilde A$ while restricting the functional spaces under consideration to functions that are symmetric about $x=0$. All the arguments proceed identically except the proof of Lemma~\ref{lem:u=1} which we may prove by simultaneously sliding two symmetric solutions of~\eqref{semilinearBall} (one on either side of $\tilde A$) instead of one. In the end, we obtain a symmetric coexistence solution $(\tilde u,\tilde v)$ in $\tilde \Omega$ and $\tilde A$ respectively. Due to the symmetry, we have $\tilde u'(0)=\tilde v'(0)=0$, so when restricted to $\Omega$ and $A$, $(\tilde u|_\Omega,\tilde v|_{A})$ yields a solution to~\eqref{eq:stat1D}.
\end{proof}
\noindent Alternatively, Theorem~\ref{thm:existence1D} could be proved using a reflection argument.

The focus of this section is the behavior of solutions to~\eqref{eq:stat1D} in the limit as $a\to 0$. One may expect that a predator domain of size zero cannot support any predators, but this is not necessarily the case. Consider, for example, fishermen fishing along the shore of a lake---the shoreline has zero area and the fishermen can still catch fish that come close to the shore. Indeed, our first result on this subject  is that as $a\to 0$, the total predator population in solutions to~\eqref{eq:stat1D} does not vanish and that the prey population density approaches a solution to
\begin{equation}\label{eq:unique_limit}
    \begin{cases}
        -d_u u''=f(u) & 0<x<L\\
        u(0)=\frac{\gamma}{\alpha}\\
        u'(L)=0.
    \end{cases}
\end{equation}
The system~\eqref{eq:unique_limit} may have many solutions. Indeed, by taking $L$ arbitrarily large, there are arbitrarily many oscillatory solutions with $0<u<\theta'$. But the following Proposition~\ref{prop:unique_limit} identifies a unique solution such that $u(L)$ is close to $1$ for $L$ large enough, and Theorem \ref{thm:thin_limit_1D} identifies this as the solution of~\eqref{eq:unique_limit} to which solutions of~\eqref{eq:stat1D} converge as $a\to 0$. These results and the remainder of the results in this section will require us to add two additional conditions for the function $f:[0,1]\to\mathbb R$ to the conditions F1-F4 in Section~\ref{sec:model}:
\begin{itemize}
    \item[(F5)] $f$ is continuously differentiable
    \item[(F6)] $f'(1)<0$.
\end{itemize}
These conditions will allow us to consider the linearization of $f$ and control the behavior of solutions when $u\approx 1$. 

\begin{proposition}\label{prop:unique_limit}
    Suppose $f$ satisfies F1-F6. There exists $\eta_0\in(0,1-\theta')$ and $b_0>0$ (at least as large as $b_0$ as in Proposition \ref{prop:subsolzet}) so that if $L>b_0$, then there exists a unique solution $u$ to \eqref{eq:unique_limit} such that $0<u<1$ and $u(L)>1-\eta_0$. Moreover, this solution $u$ is monotone increasing and $u(L)$ is increasing in $L$, and approaches $1$ as $L\to\infty$.
\end{proposition}
\noindent Equations of the form $-d_u u''=f(u)$ are well studied for a rich variety of functions $f$ (not necessarily satisfying F1-F6). Many results for such equations are found in~\cite{Kor2006,Sch2006} and references therein. The proof of Proposition~\ref{prop:unique_limit} is similar to that of earlier results (differing mainly in the mixed boundary conditions of \eqref{eq:unique_limit}), so it is relegated to Appendix \ref{app:unique_limit_proof}.

        \begin{theorem}\label{thm:thin_limit_1D}
            Suppose $\alpha>\gamma$. Fix $a_0>0$, let $b_0,\eta_0$ be as in Proposition~\ref{prop:unique_limit} and let $L_0>a_0+b_0$ be sufficiently large that $\zeta_{a_0,L_0}(L_0)>1-\eta_0$. For any $0\leq a<a_0$ and fixed $L>L_0$, let $u_a$ and $v_a$ denote a coexistence equilibrium  solution to~\eqref{eq:stat1D} guaranteed by Theorem~\ref{thm:existence1D}. Then the following hold in the limit $a\to 0$:
            \begin{itemize}
                \item[(i)] $u_a$ converges in $C^0$ to the unique solution $u$ to~\eqref{eq:unique_limit} guaranteed by Proposition~\ref{prop:unique_limit}.
                \item[(ii)] $u_a$ converges to $u$ in $C^2$ in every interval $(x_0,L]$ for $0<x_0<L$.
                \item[(iii)] $v_a$ is unbounded in $C^0$, but
                \begin{equation}
                    0<\lim_{a\to 0}\int_0^a v_a(x)\,dx=\frac{d_u\alpha}{\beta\gamma}u'(0)<\infty.
                \end{equation}
                That is, $v_a$ converges to a Dirac mass concentrated at $x=0$.
            \end{itemize}
            
        \end{theorem}
        \begin{proof}
            First, observe that
            \begin{equation}\label{eq:total_predators}
            \begin{split}
                \int_0^a v_a\,dx&=\frac{\alpha}{\gamma}\int_0^a u_av_a\,dx\\
                &=\frac{\alpha}{\beta\gamma}\int_0^a\left(d_u u_a''+f(u_a)\right)\,dx\\
                &=\frac{\alpha}{\beta\gamma}\left(-\int_a^L d_uu_a''\,dx+\int_0^af(u_a)\,dx\right)\\
                &=\frac{\alpha}{\beta\gamma}\left(\int_0^L f(u_a)\,dx\right)\\
                &\leq\frac{\alpha L}{\beta\gamma}\Vert f\Vert_{C^0}.
            \end{split}
            \end{equation}
            Thus, $v_a$ is bounded in $L^1$ in the limit $a\to 0$.

            Now we show that $u_a$ is bounded in $C^1$. We know $\zeta_{a,L}<u_a<1$, so $u_a$ is bounded in $C^0$. For $x\in(0,L)$, we calculate
            \begin{align*}
                |u_a'(x)|&\leq\int_0^x \left|-f(u_a)+\beta 1_{(0,a)}u_av_a\right|\,dx\\
                &\leq L\Vert f\Vert_{C^0}+\beta\Vert v_a\Vert_{L^1}.
            \end{align*}
            Since $v_a$ is bounded in $L^1$, $u_a$ is bounded in $C^1$. Thus, by the Arzel\`a--Ascoli Theorem, the set $\{u_a\}$ is precompact in $C^0$.

            Suppose that $u_a$ does not converge to the unique solution $u$ to~\eqref{eq:unique_limit} with $u(L)>1-\eta_0$ guaranteed by Proposition~\ref{prop:unique_limit}. Then due to the precompactness of $\{u_a\}$, there exists a sequence $a_n$ in $(0,a_0)$ and a corresponding sequence $u_n:=u_{a_n}$ converging to some $\hat u\neq u$ in $C^0$. For any $x_0\in (0,L)$, we have that for $n$ sufficiently large, $a_n<x_0$ so $-d_u u_n''=f(u_n)$ on $(x_0,L)$. Integrating twice, we may write for any $x\in (x_0,L)$,
            \begin{equation}\label{eq:u_n_representation}
                u_n(x)=u_n(x_0)+u_n'(x_0)(x-x_0)-\frac{1}{d_u}\int_{x_0}^x (x-s)f(u_n(s))\,ds.
            \end{equation}
            Since $f$ is uniformly continuous and $u_n\to \hat u$ in $C^0$, we conclude that $u_n'(x_0)$ converges to some value, say, $d\in\mathbb R$. Passing~\eqref{eq:u_n_representation} to a limit, we have
            \begin{equation}\label{eq:uhat_integral_identity}
                \hat u(x)=\hat u(x_0)+d(x-x_0)-\frac{1}{d_u}\int_{x_0}^x (x-s)f(\hat u(s))\,ds.
            \end{equation}
            We conclude that $\hat u$ is twice differentiable and satisfies
            \begin{equation}\label{eq:uhat_eqn}
                -d_u \hat u''=f(\hat u)
            \end{equation}
            on $(x_0,L)$. But since $x_0\in (0,a_0)$ is arbitrary, we conclude that $\hat u$ satisfies~\eqref{eq:uhat_eqn} on $(0,L)$.

            From~\eqref{eq:uhat_integral_identity}, we may also calculate
            \begin{align*}
                \hat u'(L)&=d-\frac{1}{d_u}\int_{x_0}^L f(\hat u(s))\,ds\\
                &=d-\lim_{n\to\infty}\frac{1}{d_u}\int_{x_0}^L f(u_n(s))\,ds\\
                &=d+\lim_{n\to\infty}\int_{x_0}^L u_n''\,ds\\
                &=d+\lim_{n\to\infty}-u_n'(x_0)\\
                &=0.
            \end{align*}

            Now we show that $\hat u(0)=\gamma/\alpha$. Let $2\delta=\hat u(0)-\gamma/\alpha$. If $\delta\neq 0$, then either $\delta>0$ or $\delta<0$. Suppose $\delta>0$. Letting $\Lambda$ be an upper bound for both $\Vert v_a\Vert_{L^1}$ and $\Vert u_a\Vert_{C^1}$, take $n$ sufficiently large that $a_n<\delta/\Lambda$. By the Mean Value Theorem, we must have $u_n>\frac{\gamma}{\alpha}+\delta$ for all $x\in (0,a_n)$. Then, with $v_n:=v_{a_n}$,
            \begin{equation*}
                0=\int_0^{a_n}\left(v_n''+(\alpha u_n-\gamma)v_n\right)>\Lambda\alpha\delta>0,
            \end{equation*}
            a contradiction. The $\delta<0$ case is exactly similar. We conclude that $\delta=0$ and thus $\hat u(0)=\frac{\gamma}{\alpha}$.

            Thus, we have shown that $\hat u$ is a solution to~\eqref{eq:unique_limit}.
            For each $n$, we know that
            \begin{equation}
                u_n(L)>\zeta_{a_n,L}(L)>\zeta_{a_0,L}(L)>\zeta_{a_0,L}(L_0)>\zeta_{a_0,L_0}(L_0)>1-\eta_0,
            \end{equation}
            from which we conclude that $1-\eta_0<\hat u(L)<1$. But due to Proposition~\ref{prop:unique_limit}, there is only one solution $u$ to~\eqref{eq:unique_limit} with $1-\eta_0<u(L)<1$, and we have assumed $\hat u\neq u$. This is a contradiction, and thus we conclude that $\lim_{a\to 0} u_a=u$ in $C^0$. We may repeat the derivation of~\eqref{eq:uhat_integral_identity} to see that $u_n$ converges to $u$ in $C^2(x_0,L)$ for each $x_0\in (0,L)$. Finally, from~\eqref{eq:total_predators}, we have in the limit as $a\to 0$,
            \begin{equation}
                \lim_{a\to 0}\int_0^a v_a\,dx=\frac{\alpha}{\beta\gamma}\int_0^L f(u)\,dx=-\frac{d_u\alpha}{\beta\gamma}\int_0^L u''\,dx=\frac{d_u\alpha}{\beta\gamma}u'(0).
            \end{equation}
            Integrating~\eqref{eq:unique_limit} from $0$ to $L$, we see that
            \begin{equation*}
                u'(0)=\sqrt{\frac{2}{d_u}F(u(L)}.
            \end{equation*}
            Since $u(L)>1-\eta_0>\theta'$, each of $F(u(L))$, $u'(0)$, and $\lim_{a\to 0}\int_0^a v_a\,dx$ are positive.
        \end{proof}

        From Theorem~\ref{thm:thin_limit_1D}, we see that as $a\to 0^+$, the total predator population of coexistence equilibria for~\eqref{eq:higher_dimensional_model} with $\Omega=(0,L)$ (for large $L$) and $A=(0,a)$ approaches a finite, positive limit. This holds even if there is not a unique coexistence equilibrium $(u_a,v_a$) for each $a$ sufficiently small, but if there is a sufficiently regular family of such equilibria, we could compute an asymptotic expansion for the total predator population:
        \begin{equation*}
        V(a)=\int_0^a v_a\,dx=V(0)+V'(0)a+O(a^2).
        \end{equation*}
        Theorem~\ref{thm:thin_limit_1D} already derives the constant term $V(0)$, so it remains to derive the linear coefficient $V'(0)$. We will do this using formal asymptotic expansions, assuming sufficient regularity of the map $a\mapsto (u_a,v_a)$. In fact, this regularity is limited as $u_a$ is of necessity not twice differentiable at $x=a$. Instead, we will consider regularity of $u_a$ in $(0,a)$ and $(a,L)$ separately. To facilitate this, we perform a change of coordinates for solutions $u_a$ and $v_a$. Letting $u_1=u_a\big|_{(0,a)}$ and $u_2=u_a\big|_{(a,L)}$, we apply the coordinate transformation $x\to x/a$ to $u_1$ and $v_a$ and the coordinate transformation $x\to (x-a)/(L-a)$ to $u_2$ to obtain equations for $u_1$, $u_2$, and $v=v_a$ (we drop the subscript $a$ for convenience) in these new coordinates:            \begin{equation}\label{eq:higher_dimensional_model_new_coordinates}
            \begin{cases}
                \frac{d_u}{a^2} u_1''+f(u_1)-\beta u_1 v=0 & 0<x<1\\
                \frac{d_u}{(L-a)^2} u_2''+f(u_2)=0 & 0<x<1\\
                \frac{d_v}{a^2} v''-\gamma v+\alpha u_1v=0 & 0<x<1\\
                u_1'(0)=u_2'(1)=v'(0)=v'(1)=0 & \\
                u_1(1)=u_2(0) & \\
                \frac{1}{a}u_1'(1)=\frac{1}{L-a}u_2'(0). & 
            \end{cases}
        \end{equation}
            We consider the following ans\"atze for the asymptotic expansions of $u_1$, $u_2$, and $v$ in small $a$, informed by Theorem~\ref{thm:thin_limit_1D}:        
            \begin{equation}\label{eq:small_a_ansatze}
                \begin{split}
                    u_1(x)&=\frac{\gamma}{\alpha}+a u_{1,1}(x)+a^2u_{1,2}(x)+O(a^3)\\
                    u_2(x)&=w_1(Lx)+a u_{2,1}(x)+O(a^2)\\
                    v(x)&=\frac{1}{a}\frac{d_u\alpha}{\beta\gamma}w_1'(0)+v_0(x)+a v_1(x)+a^2v_2(x)+O(a^3),
                \end{split}
            \end{equation}
            Where $w_1$ is the solution to~\eqref{eq:unique_limit} guaranteed by Proposition~\ref{prop:unique_limit}. Note that the expansion for $v$ leads in order $a^{-1}$ since we expect the predator density to blow up as $a\to 0$ while the total predator population approaches a positive constant due to Theorem~\ref{thm:thin_limit_1D}. 

            Substituting~\eqref{eq:small_a_ansatze} into~\eqref{eq:higher_dimensional_model_new_coordinates} and comparing terms of like order in $a$, we see that $u_{1,1}$ solves
            \begin{equation}
                \begin{cases}
                    u_{1,1}''=w_1'(0) & 0<x<1\\
                    u_{1,1}'(0)=0 & \\
                    u_{1,1}'(1)=w_1'(0)
                \end{cases}
            \end{equation}
            We observe that the solvability condition for this Neumann problem is already exactly met. We can determine $u_{1,1}$ up to an additive constant:
            \begin{equation}
                u_{1,1}(x)=\frac{1}{2}w_1'(0)x^2+c.
            \end{equation}
            The value of $c$ is determined by the solvability condition for $v_2$ which solves
            \begin{equation}
                \begin{cases}
                    d_vv_2''=-\frac{d_u\alpha^2}{\beta\gamma}w_1'(0) u_{1,1} & 0<x<1\\
                    v_2'(0)=v_2'(1)=0.
                \end{cases}
            \end{equation}
            We must have $\int_0^1 v_2''\,dx=-\frac{d_u\alpha^2}{\beta\gamma}w_1'(0)\int_0^1u_{1,1}\,dx=0$. Thus, we find
            \begin{equation}
                u_{1,1}(x)=\frac{1}{2}w_1'(0)\left(x^2-\frac{1}{3}\right).
            \end{equation}
            Next we use this solution for $u_{1,1}$ to find that $u_{2,1}$ and $u_{1,2}$ solve
            \begin{equation}\label{eq:u21}
                \begin{cases}
                    \frac{d_u}{L^2}u_{2,1}''+u_{2,1}f'\left(w_1(Lx)\right)=\frac{2}{L}f\left(w_1(Lx)\right) & 0<x<1\\
                    u_{2,1}'(1)=0 & \\
                    u_{2,1}(0)=\frac{1}{3}w_1'(0),
                \end{cases}
            \end{equation}
            and
            \begin{equation}
                \begin{cases}
                    d_u u_{1,2}''=-f\left(\frac{\gamma}{\alpha}\right)+\frac{\beta\gamma}{\alpha} v_0+\frac{\alpha}{2L\gamma}\left(w_1'(0)\right)^2\left(x^2-\frac{1}{3}\right) & 0<x<1\\
                    u_{1,2}'(0)=0 & \\
                    u_{1,2}'(1)=\frac{1}{L}u_{2,1}'(0).
                \end{cases}
            \end{equation}
            Thus, $v_0$ is determined by the solvability condition
            \begin{equation}
                \frac{1}{L}u_{2,1}'(0)=\int_0^1 u_{1,2}''\,dx=\frac{1}{d_u}\left(\frac{\beta\gamma}{\alpha}v_0-f\left(\frac{\gamma}{\alpha}\right)\right),
            \end{equation}
            leading to
            \begin{equation}\label{eq:v0}
                v_0=\frac{\alpha}{\beta\gamma}\left(\frac{d_u}{L}u_{2,1}'(0)+f\left(\frac{\gamma}{\alpha}\right)\right).
            \end{equation}
            Using~\eqref{eq:v0}, and letting $w_2$ be $u_{2,1}$ represented in our original coordinates, we obtain
            \begin{equation}\label{eq:V_exp}
                V(a)=\int_0^a v_a(x)\,dx=\frac{d_u\alpha}{\beta\gamma}w_1'(0)+\frac{\alpha}{\beta\gamma}\left(d_u w_2'(0)+f\left(\frac{\gamma}{\alpha}\right)\right)a+O(a^2),
            \end{equation}
            where $w_1$ is the solution to~\eqref{eq:unique_limit} guaranteed by Proposition~\ref{prop:unique_limit} and $w_2$ is a solution to
            \begin{equation}\label{eq:w2}
                \begin{cases}
                    d_uw_2''+w_2 f'(w_1)=\frac{2}{L}f(w_1) & 0<x<L\\
                    w_2'(L)=0 & \\
                    w_2(0)=\frac{1}{3}w_1'(0).
                \end{cases}
            \end{equation}

        Depending on the parameters $\alpha$, $\beta$, $\gamma$, $f$, $d_u$, $d_v$, and $L$, the value of $V'(a)$ may be positive or negative. The following Theorem shows that when $L$ is large, the value of $V'(a)$ approximately proportional to $f(\gamma/\alpha)$.
        \begin{theorem}\label{thm:thin_limit_slope}
            In the large $L$ limit, the linear coefficient in~\eqref{eq:V_exp} is
            \begin{equation}\label{eq:linear_term_asymptotics}
                \lim_{L\to\infty}V'(0)=\lim_{L\to\infty}\frac{\alpha}{\beta\gamma}\left(d_u w_2'(0)+f\left(\frac{\gamma}{\alpha}\right)\right)=\frac{2}{3}\frac{\alpha}{\beta\gamma}f\left(\frac{\gamma}{\alpha}\right).
            \end{equation}
        \end{theorem}
        \begin{proof}
            Let $w_1$ be the monotone solution to~\eqref{eq:unique_limit} in $[0,L]$ guaranteed by Proposition~\ref{prop:unique_limit} for large $L$. Let $w_2$ be the solution to~\eqref{eq:w2}. We begin by showing that $w_2(L)$ vanishes as $L\to\infty$. We may differentiate $d_u w_1''+f(w_1)=0$ to see that $w_1'$ solves the linear equation
            \begin{equation}\label{eq:w1_prime}
               d_u (w_1')''+f'(w_1)w_1'=0.
            \end{equation}
            Next, we will introduce two auxiliary functions. Define
            \begin{equation*}
                z(x)=w_2(x)+\left(\frac{x}{L}-\frac{1}{3}\right)w_1'(x)
            \end{equation*}
            Then $z$ solves
            \begin{align*}
                d_uz''+f'(w_1)z&=d_uw_2''+d_u\left(\frac{x}{L}-\frac{1}{3}\right)w_1'''+\frac{2d_u}{L}w_1''+f'(w_1)w_2+\left(\frac{x}{L}-\frac{1}{3}\right)f'(w_1)w_1'\\
                &=\left(d_uw_2''+f'(w_1)w_2\right)+\left(\frac{x}{L}-\frac{1}{3}\right)\left(d_u w_1'''+f'(w_1)w_1'\right)-\frac{2}{L}f(w_1)\\
                &=0.
            \end{align*}
            We also have $z(0)=w_2(0)-w_1'(0)/3=0$ and
            \begin{equation*}
                z'(L)=w_2'(L)+\frac{1}{L}w_1'(L)+\frac{2}{3}w_1''(L)=-\frac{2}{3}\frac{f(w_1(L))}{d_u}.
            \end{equation*}
            Our next auxiliary function is $y$ given by
            \begin{equation*}
                y=\frac{\partial w_1}{\partial L}.
            \end{equation*}
            Differentiating $d_u w_1''+f(w_1)=0$ with respect to $L$, we see that that $y$ solves the same equation as $w_1'$ and $z$: $-d_uy''+f'(w_1)y=0$. For boundary conditions, we have $y(0)=0$ since $w_1(0)=\gamma/\alpha$ independent of $L$. We may differentiate $w_1'(L)=0$ with respect to $L$ (using the chain rule) to obtain $y'(L)+w_1''(L)=0$, from which we conclude
            \begin{equation*}
                y'(L)=\frac{f(w(L))}{d_u}.
            \end{equation*}
            
            Since $y$ and $z$ both satisfy the same linear differential equation and $y(0)=z(0)=0$, we conclude that they are scalar multiples of each other. From their boundary conditions at $L$, we see that $z=-\frac{2}{3}y$, and therefore
            \begin{equation}
                w_2(L)=w_2(L)+\frac{2}{3}w_1'(L)=z(L)=-\frac{2}{3}y(L).
            \end{equation}
            Thus, to show that $w_2(L)$ is bounded, it suffices to show that $y(L)$ is bounded. Denote $w_1(L)$ by $q$ as in the proof of Proposition~\ref{prop:unique_limit} (in Appendix \ref{app:unique_limit_proof}). Since $q$ depends on $L$, we may write it as $q=q(L)$. Then we have $y(L)=q'(L)$. In the proof of Proposition~\ref{prop:unique_limit}, we see that for large $L$, this function $q(L)$ is invertible and we may write $L=\mathcal L(q)$ as~\eqref{eq:L_map}. Moreover, we also showed in this same proof that for sufficiently large $q$ (or equivalently, sufficiently large $L$), $\mathcal L$ is increasing and
            \begin{equation*}
                \lim_{q\to 1}\mathcal L(q)=\lim_{q\to 1}\mathcal L'(q)=\infty.
            \end{equation*}
            Therefore, 
            \begin{equation*}
                \lim_{L\to\infty}-\frac{3}{2}w_2(L)=\lim_{L\to\infty}y(L)=\lim_{L\to\infty}q'(L)=\frac{1}{\lim_{q\to 1}\mathcal L'(q)}=0.
            \end{equation*}
        
            Now, we derive the asymptotic value of $w_2'(0)$ for large $L$. Differentiating $d_u w_1''+f(w_1)=0$, we see that $w_1'$ satisfies
            \begin{equation}\label{eq:w1_differentiated}
                d_u w_1'''+f'(w_1)w_1'=0.
            \end{equation}
            Multiplying~\eqref{eq:w2} by $w_1'$ and multiplying~\eqref{eq:w1_differentiated} by $w_2$, and subtracting the second from the first, we have
            \begin{equation}
                d_u(w_2''w_1'-w_1'''w_2)=\frac{2}{L}f(w_1)w_1'
            \end{equation}
            Integrating by parts from $0$ to $L$, we have
            \begin{equation}
                d_u\left(w_2'(L)w_1'(L)-w_2'(0)w_1'(0)-w_1''(L)w_2(L)+w_1''(0)w_2(0)\right)=\frac{2}{L}(F(w_1(L))-F(\gamma/\alpha)).
            \end{equation}
            We may simplify the left hand side while (since $F$ is bounded), the right hand vanishes as $L\to\infty$:
            \begin{equation}
                \lim_{L\to\infty}\left(-d_u w_2'(0)w_1'(0)+f(w_1(L))w_2(L)-\frac{1}{3}f\left(\gamma/\alpha\right)w_1'(0)\right)=0.
            \end{equation}
            Since both $f(w_1(L))$ and $w_2(L)$ vanish as $L\to\infty$ and $w_1'(0)$ is bounded away from $0$ (as it approaches $\sqrt{2F(1)}$ as $L\to\infty$), we conclude that
            \begin{equation*}
                \lim_{L\to\infty}w_2'(0)=-\frac{1}{3d_u}f\left(\frac{\gamma}{\alpha}\right).
            \end{equation*}
            Substituting this into~\eqref{eq:V_exp}, the result is proved.
        \end{proof}

        The surprising result of Theorem \ref{thm:thin_limit_slope} is that the total predator population may actually be higher at $a=0$ than for small $a>0$. This means that if the predators are very efficient ($\gamma/\alpha<\theta$), then they are better off restricting themselves to a small territory as $V'(a)<0$. Indeed, if highly efficient predators occupied a small interval $(0,a)$ with $a>0$, then they would consume most of the prey immediately when it diffuses into the predation zone, leaving most of the interval $(0,a)$ devoid of prey, and therefore devoid of predators as well. Consequently, the main effect of occupying an interval of length $a$ as opposed to a single point at $x=0$ is to effectively decrease the size of the prey domain, limiting the amount of prey available and decreasing the total predator population at equilibrium. We will examine this phenomenon numerically in Section~\ref{sec:numerics}.

\section{Further Properties and Open Problems}\label{sec:numerics}

We have explored several properties of solutions to the model~\eqref{eq:higher_dimensional_model} analytically. There are many further properties of solutions left to be studied. In this section, we demonstrate several such properties numerically and provide conjectures to motivate future analytical studies. The purpose of this section is not to provide a rigorous numerical analysis, but rather to illustrate and complement the analytical results of the previous sections. We use numerical simulations to (i) visualize coexistence equilibria, (ii) explore how the size of the exclusion zone influences solution behavior, and (iii) highlight dynamical phenomena that suggest directions for future analytical study.

\subsection{Computations of equilibria 1D}

Of principal interest is how solutions to~\eqref{eq:higher_dimensional_model} depend on $A$ and $\Omega$. Theorems~\ref{thm:existence1D} and~\ref{thm:existenceND} give conditions for existence of coexistence solutions, and in Section~\ref{sec:thin_limit_1D}, we consider the asymptotic case when $A$ is small. But to explore the dependence of solutions on $A$ and $\Omega$ in generality, we turn to numerical solutions. So that this dependence may be understood in the simplest possible setting, we solve~\eqref{eq:higher_dimensional_model} numerically in 1D. Specifically, we consider the setting where $\Omega=(0,L)$ for some $L>0$ and $A=(0,a)$ for some $0<a<L$. Then the question of how solutions depend on $A$ and $\Omega$ reduces simply to the question of how they depend on the two numbers $a$ and $L$. However, even in this simple setting, we see that the model~\eqref{eq:higher_dimensional_model} gives rise to a remarkable variety of interesting solution behaviors. We show many such behaviors in figures in this section. Table~\ref{tab:params} lists the parameters used to generate the solutions shown in each figure. The bistable growth function $f:[0,1]\to\mathbb R$ we use for these simulations is
\begin{equation}
    f(s)=rs\left(\frac{s}{\theta}-1\right)\left(1-s\right),
\end{equation}
with the values of $r$ and $\theta$ varying between simulations, and shown in Table~\ref{tab:params}.

For these numerical simulations, we solved the time-dependent reaction--diffusion system using a method-of-lines approach. Space was discretized by finite differences on two meshes---one uniform mesh on $[0,a]$, and another uniform mesh on $[a,L]$, allowing the grid to conform to the interface $x=a$. The prey equation was discretized on the full domain $[0,L]$ while the predator equation was discretized only on $[0,a]$. Homogeneous Neumann boundary conditions for both equations were incorporated directly into the discrete Laplacian, and the interface $x=a$ was handled using the corresponding nonuniform finite-difference stencil. This yielded a sparse system of ODEs, which were integrated using SciPy's implicit Radau solver, facilitated by the analytic Jacobian of the discretized system. Simulations were initialized from spatially homogeneous data with prey density $u\equiv 1$ and predator density $v\equiv 0.5$ on the predator domain. Total prey and predator populations were computed by numerical quadrature of the discrete solutions. When computing the ``terminal population'', the average of total population over the last quarter of the simulation window was taken. These computations are intended to illustrate the analytical results, visualize representative solution behavior, and explore qualitative phenomena suggested by the model. A link to a repository containing the code used for the simulations in this section is found in the Data Declaration section below.

To begin, we examine the long-time behavior of solutions $v(t,x)$ and $u(t,x)$. We observe five types of long-time behavior.
\begin{enumerate}
    \item \textbf{Coexistence equilibrium.} Solutions may approach a locally stable coexistence equilibrium $(\hat u(x),\hat v(x))$ as $t\to\infty$. Such an equilibrium solution is shown in Figure~\ref{fig:equilibrium}. Convergence to coexistence equilibria occurs in many parameter regimes, but seems to occur consistently if the length $L-a$ of the exclusion zone is large enough and if the initial prey density is large enough to avoid immediate extinction.
    \item \textbf{Extinction.} Predators may drive the prey to extinction, and then die off themselves. This occurs when the initial prey density is sufficiently small and the initial predator density is sufficiently large that, on a timescale short relative to the characteristic predator timescale $1/\gamma$, the predators drive the prey beneath their Allee threshold. This is difficult to achieve if the length $L-a$ of the exclusion zone is large because the prey in the exclusion zone cannot be directly controlled by the predators. But as we shall see, when $L-a$ is small, numeric solutions tend to converge to the trivial equilibrium.
    \item\textbf{Prey-only.} If $\gamma>\alpha$, then solutions may approach the non-coexistence equilibrium $(u,v)=(1,0)$ (or solutions may also approach extinction in this case).
    \item\textbf{Limit cycle.} Solutions may approach a limit cycle when the timescale for the decrease of the prey population due to predation in $(0,a)$ is approximately equal to the timescale for the prey population to be replenished from diffusion from the exclusion zone. A periodic solution representing such a limit cycle is shown in Figure~\ref{fig:limit_cycle}. A plot of the total populations
    \begin{equation}
        U(t)=\int_0^L u(t,x)\,dx,\quad V(t)=\int_0^a v(t,x)\,dx
    \end{equation}
    is shown in Figure~\ref{fig:limit_cycle_total_population}, where we observe that $U(t)$ and $V(t)$ become asymptotically periodic as $t\to\infty$. These periodic solutions are characterized by waves of prey traveling from the exclusion zone into the predators' territory.
    \item\textbf{Chaotic behavior.} If the length $a$ of the predator zone is larger than the multiple waves of prey can be present simultaneously (but still such that $L-a$ is large enough to allow the prey to persist), these waves may interfere leading to apparently irregular (and possibly chaotic) behavior. We hypothesize that as $a$ increases, the periodic solutions undergo a cascade of period-doubling bifurcations leading eventually to a chaotic attractor. A more detailed study of this phenomenon will be presented in forthcoming work.
\end{enumerate}

\begin{figure}[h]
  \centering
  \begin{subfigure}{0.48\linewidth}
    \centering
    \includegraphics[width=\linewidth]{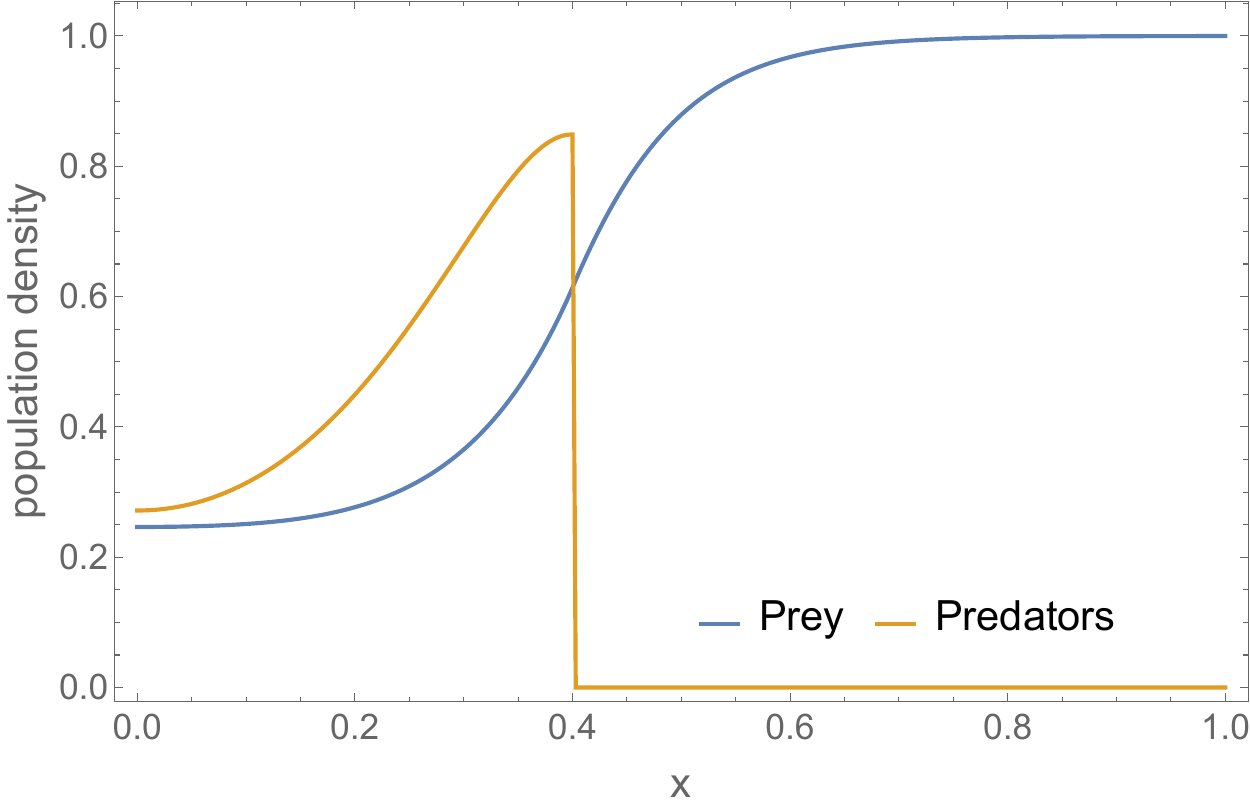}
    \caption{An approximation of an equilibrium solution to~\eqref{eq:higher_dimensional_model} with $\Omega=(0,1)$ and $A=(0,0.4)$.}
    \label{fig:equilibrium}
  \end{subfigure}
  \hfill
  \begin{subfigure}{0.48\linewidth}
    \centering
    \includegraphics[width=\linewidth]{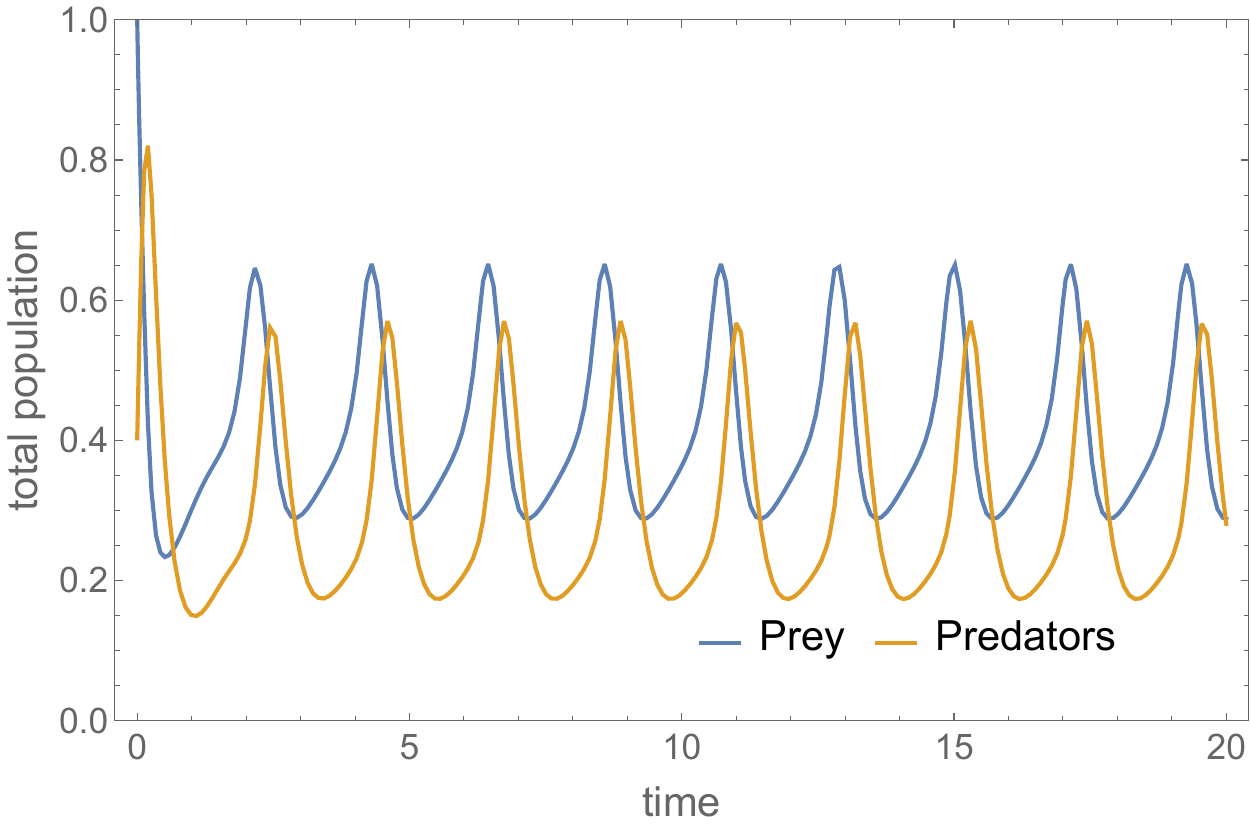} 
    \caption{Total populations of a solution to~\eqref{eq:higher_dimensional_model} approach a limit cycle.}
    \label{fig:limit_cycle_total_population}
  \end{subfigure}
  \caption{In some parameter regimes, solutions to~\eqref{eq:higher_dimensional_model} approach a coexistence equilibrium (left). In others, solutions approach a limit cycle (right).}
  \label{fig:nontrival orbits}
\end{figure}

\begin{figure}[h]
  \centering
  \includegraphics[width=0.9\textwidth]{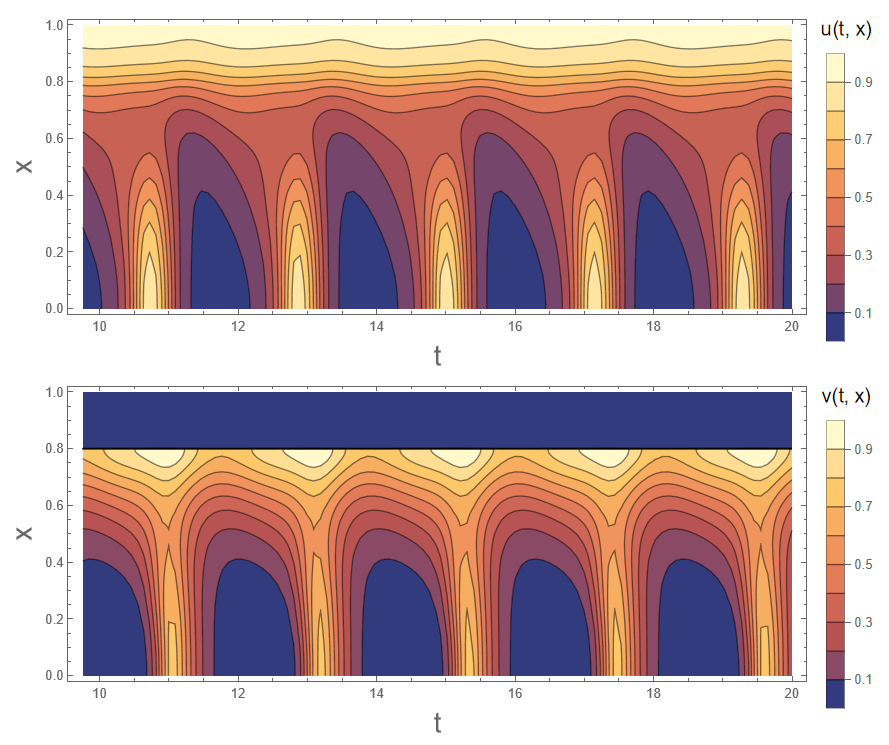}
  \caption{These contour plots show the population densities of prey (top) and predators (bottom) depending on time and space as solutions of~\eqref{eq:higher_dimensional_model} approach a limit cycle.}
  \label{fig:limit_cycle}
\end{figure}


Fixing an initial condition (say, $u(0,x)=v(0,x)=1$) and the parameter values $L$, $\alpha$, $\beta$, $\gamma$, $\theta$, $r$, $d_u$, and $d_v$, we can observe the long-time behavior of solutions to~\eqref{eq:higher_dimensional_model} as the remaining parameter $a$ varies between $0$ and $L$. The following six quantities (three each for $u$ and $v$) are useful in evaluating the long-time behavior of solutions:
\begin{align}
    \widehat U(a)&=\limsup_{t\to\infty}\int_0^L u(t,x)\,dx\quad\quad &&\widehat V(a)=\limsup_{t\to\infty}\int_0^a v(t,x)\,dx\\
    \widecheck U(a)&=\liminf_{t\to\infty}\int_0^L u(t,x)\,dx\quad\quad &&\widecheck V(a)=\liminf_{t\to\infty}\int_0^a v(t,x)\,dx\\
    \overline U(a)&=\lim_{t\to\infty}\frac{1}{T}\int_0^L u(t,x)\,dx\quad\quad && \overline V(a)=\lim_{t\to\infty}\frac{1}{T}\int_0^a v(t,x)\,dx
\end{align}
These are respectively the maximum, minimum, and average of the populations of solutions after a long time. We refer to the triples $(\widehat U(a),\overline{U}(a),\widecheck{U}(a))$ and $(\widehat V(a),\overline{V}(a),\widecheck{V}(a))$ as the \emph{limiting profiles} of $u$ and $v$ respectively. If solutions $u(t,x)$ and $v(t,x)$ converge to an equilibrium, then
\begin{equation}
    \begin{split}
        \widecheck U(a)&=\overline U(a)=\widehat U(a)\\
        \widecheck V(a)&=\overline V(a)=\widehat V(a).
    \end{split}
\end{equation}
But if solutions approach a limit cycle or display chaotic behavior, then 
\begin{equation}
    \begin{split}
        \widecheck U(a)&<\overline U(a)<\widehat U(a)\\
        \widecheck V(a)&<\overline V(a)<\widehat V(a), 
    \end{split}
\end{equation}
with $\widehat U(a)-\widecheck U(a)$ and $\widehat V(a)-\widecheck V(a)$ being the amplitude of the oscillations. The limiting profiles demonstrate how the long-time behavior of solutions depends on $a$ and what types of bifurcations occur. The results are as rich and varied as families of solutions $u(t,x)$ and $v(t,x)$ themselves. In a future work, we will more fully explore this variety. Here, we highlight three particular cases that illustrate a few of the key features of these limiting profiles.
\begin{enumerate}
    \item Figure~\ref{fig:limiting_profile_1} shows three features frequently seen in limiting profiles:
    \begin{itemize}
        \item \textbf{Interior maximum.} The limiting total predator population $\overline V(a)$ is maximized at a point $a_{\text{max}}$ in the interior of the interval $(0,L)$, meaning that the predators do best when they occupy some, but not all of the prey's territory.
        \item \textbf{Hopf bifurcation.} There is a value $a_{\text{hopf}}$ such that for $a<a_{\text{hopf}}$, solutions appear to approach coexistence equilibria while for $a>a_{\text{hopf}}$, solutions appear to approach limit cycles, until $a$ becomes large enough for another bifurcation to occur. This suggests that a Hopf bifurcation is present.
        \item \textbf{Sudden extinction.} At $a=a_{\text{ext}}$, the periodic behavior suddenly vanishes and solutions instead approach extinction (possibly a saddle-node bifurcation of limit cycles occurs here).
    \end{itemize}
    \item Figure~\ref{fig:limiting_profile_3} shows two more features sometimes seen in limiting profiles.
    \begin{itemize}
        \item \textbf{Local maximum in the thin-limit.} The limiting total predator approaches a local maximum in the thin-limit $a\to 0^+$. This was predicted in Section~\ref{sec:thin_limit_1D} due to the indeterminate sign of $v_0$ in~\eqref{eq:v0}. This reflects the result of Theorem \ref{thm:thin_limit_slope}, which predicts that the predator population is locally maximized in the limit $a\to 0$ if $\gamma/\alpha<\theta$.
        \item \textbf{Global maximum at extinction.} Remarkably, after reaching a local minimum, $\overline V(a)$ increases right up until the point of saddle-node bifurcation of limit cycles at $a_{\text{ext}}$. Thus, $a_{\text{max}}$ appears to coincide with $a_{\text{ext}}$ (or, at least these two values are within the numerical step size used for $a$ in this computation).  
    \end{itemize}
    \item Figure~\ref{fig:limiting_profile_4} shows two final features of interest:
    \begin{itemize}
        \item \textbf{Global maximum in the thin-limit.} Not only does $\overline V(a)$ approach a local maximum as $a\to 0^+$, but this is actually the global maximum. This suggests that the best situation for the predators is if they cluster together, leaving the largest possible exclusion zone for the prey.
        \item \textbf{No Hopf bifurcation.} No limit cycles appear for these parameter values, so the bifurcation at $a_{\text{ext}}$ is a saddle-node bifurcation of equilibria (rather than of limit cycles).
    \end{itemize}
\end{enumerate}

\begin{figure}[h]
  \centering
  \includegraphics[width=0.9\textwidth]{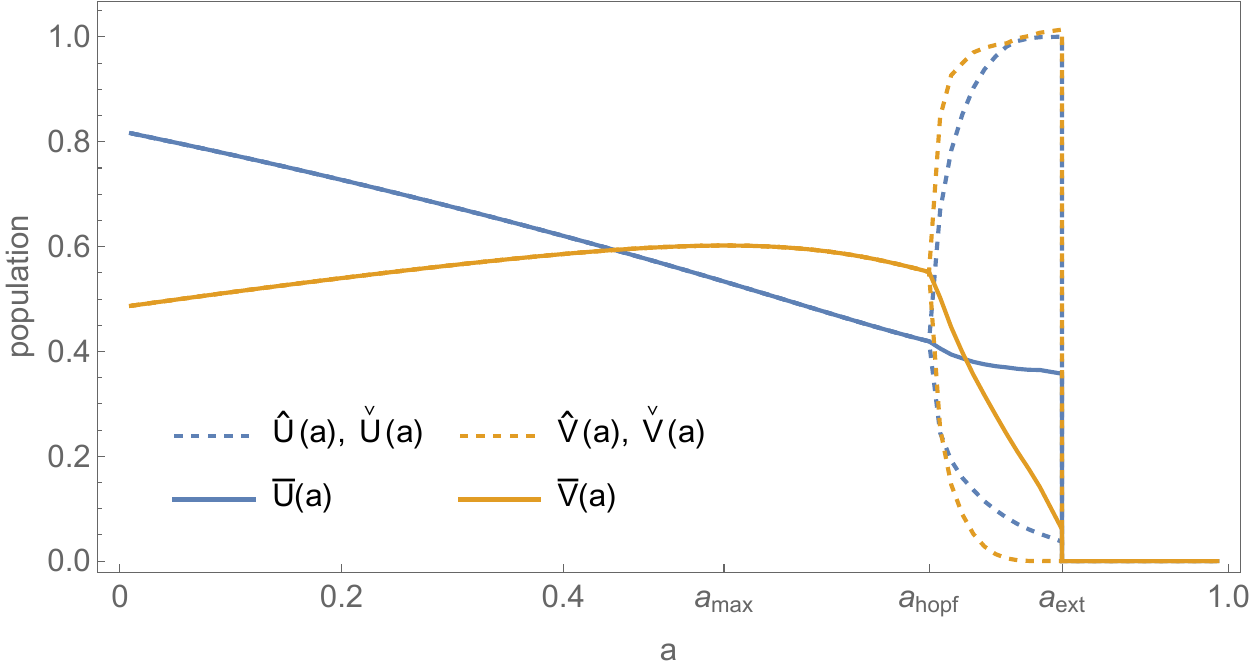}
  \caption{A plot of the limiting profiles exhibiting an interior maximum, Hopf bifurcation, and extinction.}
  \label{fig:limiting_profile_1}
\end{figure}

\begin{figure}[h]
  \centering
  \includegraphics[width=0.9\textwidth]{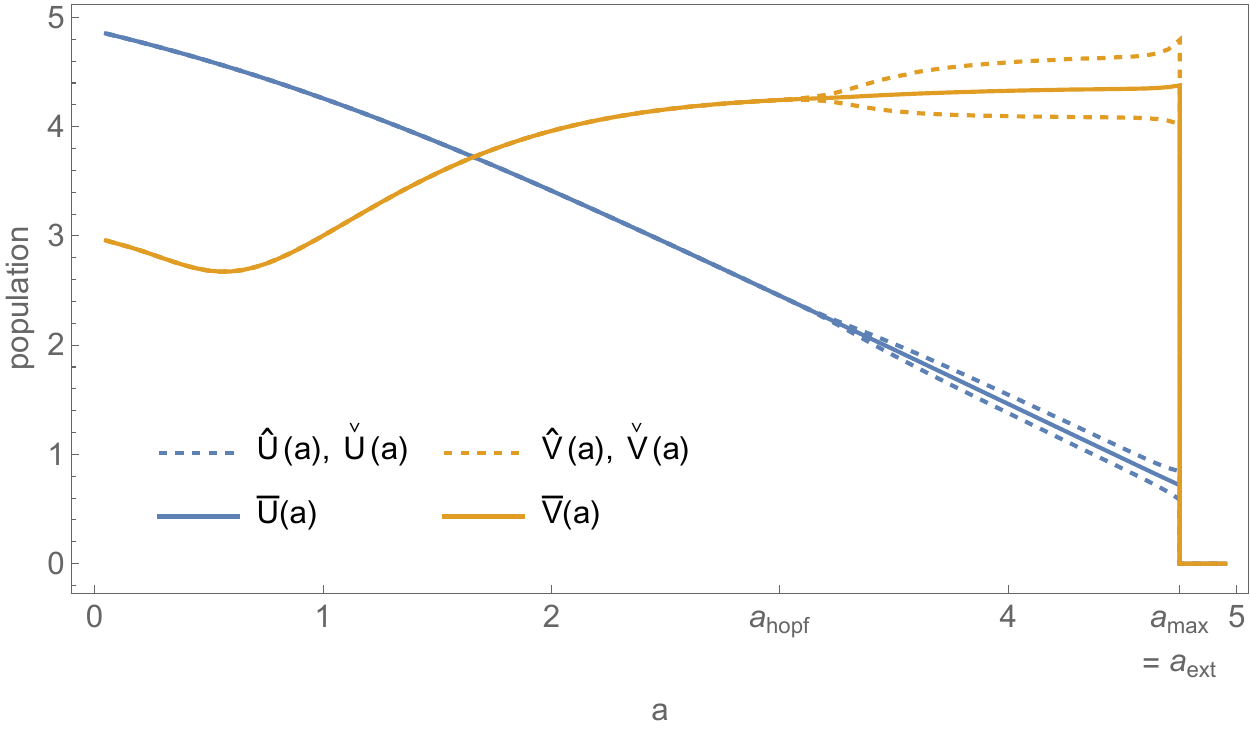}
  \caption{A plot of the limiting profiles exhibiting a local maximum in the thin-limit, a Hopf bifurcation and a global maximum at extinction.}
  \label{fig:limiting_profile_3}
\end{figure}

\begin{figure}[h]
  \centering
  \includegraphics[width=0.9\textwidth]{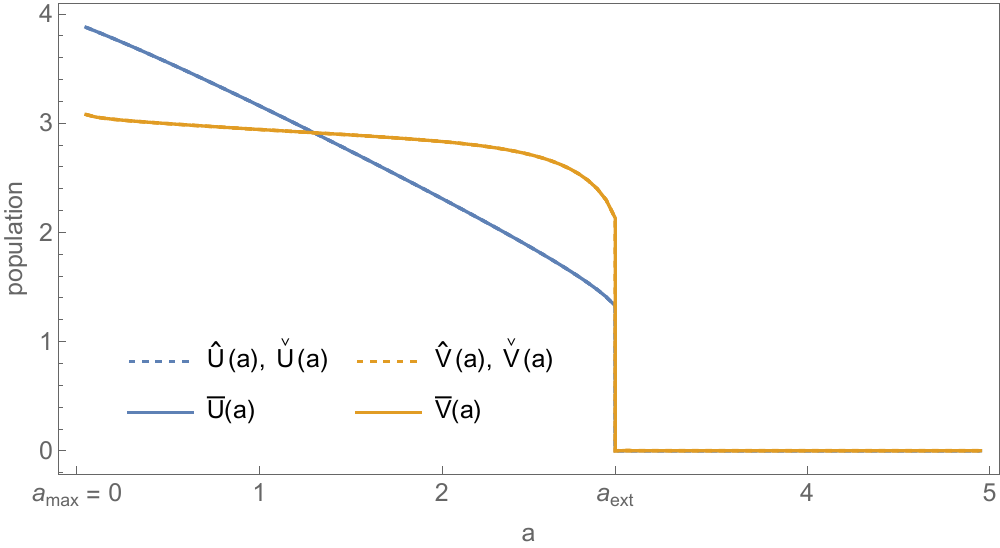}
  \caption{A plot of the limiting profiles exhibiting a global maximum in the thin-limit and no Hopf bifurcation.}
  \label{fig:limiting_profile_4}
\end{figure}

\begin{table}[t]
  \centering
  \scriptsize
  \setlength{\tabcolsep}{3.5pt}
  \renewcommand{\arraystretch}{1.15}
  \begin{tabularx}{\linewidth}{@{}l*{10}{>{\centering\arraybackslash}X}@{}}
    \toprule
    \textbf{Figure} & $L$ & $a$ & $\alpha$ & $\beta$ & $\gamma$ & $\theta$ & $r$ & $d_u$ & $d_v$ \\
    \midrule
   ~\ref{fig:equilibrium} & 1 & 0.4 & 14 & 12 & 5 & 0.05 & 1 & 0.1 & 0.05 \\
   ~\ref{fig:limit_cycle_total_population},~\ref{fig:limit_cycle} & 1 & 0.8 & 14 & 12 & 5 & 0.05 & 1 & 0.1 & 0.05 \\
   ~\ref{fig:limiting_profile_1} & 1 & $0<a<L$ & 13.9 & 10 & 5 & 0.04 & 0.904 & 1 & 0.52 \\
   ~\ref{fig:limiting_profile_3} & 5 & $0<a<L$ & 3 & 3 & 0.9 & 0.3 & 30 & 1 & 1 \\
   ~\ref{fig:limiting_profile_4} & 5 & $0<a<L$ & 1 & 3 & 0.05 & 0.3 & 1 & 1 & 1\\
    \bottomrule
  \end{tabularx}
  \caption{Parameter values used to generate each figure.}
  \label{tab:params}
\end{table}

We conclude this section emphasizing three features we observe in each of the limiting profiles shown in Figures~\ref{fig:limiting_profile_1}--\ref{fig:limiting_profile_4}.
\begin{itemize}
    \item In every case examined, $\overline U(a)$ is decreasing in $a$, indicating that increasing $a$ is always detrimental to the prey. On the other hand, the predators may do best in a small territory, a large territory, or an intermediate territory depending on the parameter values. It is expected that the situation in higher dimensions is even more complicated, with the limiting predator population depending not only on territory size but also shape. Even in one dimension, the question remains whether or not more predators can be supported within a single interval of length $a_{\text{max}}$, or several disjoint intervals whose lengths sum to $a_{\text{max}}$ with exclusion zones between them. These questions will be the topics of further research.
    \item For small enough $a$, solutions are consistently observed to converge to coexistence equilibria (assuming the initial prey density is large enough to avoid extinction), and the limiting total population does not vanish in the limit of small $a$:
    \begin{equation*}
        \lim_{a\to 0^+}\overline V(a)>0.
    \end{equation*}
    This indicates that population density blows up as $a\to 0^+$ while total population approaches a finite limit. This supports the thin-limit analysis of Section~\ref{sec:thin_limit_1D}. In fact, the formula~\eqref{eq:V_exp} allows us to directly calculate the limit of $\overline V(a)$ and its derivative as $a\to 0^+$.
    \item Each of the limiting profiles shown demonstrate that solutions approach equilibrium for $a$ large enough. In each case, the transition from nontrivial limiting behavior to extinction is \emph{sudden}, and appears to occur via a saddle-node bifurcation (either a saddle-node of limit cycles or of equilibria). Thus, in a situation where a predator--prey system is near the extinction threshold $a_{\text{ext}}$, a small perturbation of the exclusion zone size can, without warning, lead to the catastrophic extinction of both species. This is a particularly alarming possibility for the scenario shown in Figure~\ref{fig:limiting_profile_3} because the optimal value of $a$ for the predators is located so precariously close to the extinction threshold.
\end{itemize}
These numerical experiments reveal the rich diversity of behaviors this simple model exhibits. In further work, we will conduct a more thorough numerical investigation of these effects and provide rigorous analytical proofs of some of the phenomena observed numerically.

\section{Discussion}

We have explored a diffusive predator--prey model with an \emph{exclusion zone} where the predators cannot enter. This model is relevant to a variety of ecological scenarios, such as the development of marine protected areas (MPAs) where fishing is excluded from certain areas~\cite{botsford, edgar}, landscapes where property boundaries allow small-bodied species to pass but not larger ones~\cite{loock}, and situations where predator territoriality and movement behaviors concentrate their activity spatially~\cite{moorcroft, catdog}. A similar model was studied by~\cite{DuShi2006}, but here we expand on that work to consider scenarios in which the prey species is subject to an Allee effect~\cite{courchamp, stephens}. More specifically, we considered what ecologists term a \emph{strong Allee effect} in which the growth rate of the prey population is negative for a range of small densities~\cite{stephens}. Review papers characterize the biological evidence for Allee effects in natural and experimental settings~\cite{kramer, muir}.


Holding other parameters fixed, the conditions for coexistence of predator and prey are somewhat relaxed when $\Omega\setminus A$ is sufficiently large, demonstrating the value of the exclusion zone. Interestingly, the transition from coexistence to mutual extinction can be quite abrupt as seen in Figures~\ref{fig:limiting_profile_1}-\ref{fig:limiting_profile_4}. Specifically, for many scenarios in which $a$ is close to $L$ (i.e., when the exclusion zone is small), we observe an apparent discontinuity in which the long-term average population densities of both predators and prey transition from positive values to zero as $a$ passes through a threshold. This sharp transition would be a concern in the context of MPAs because it implies that fisheries could collapse quickly, without much warning, as a consequence of changes in their spatial management. That is, merely changing the location of the boundary of an MPA could be sufficient to drive a well-functioning fisher-fisheries system to collapse, if that boundary change resulted in the MPA being too small. A similarly abrupt collapse to mutual extinction would occur if the MPA were just barely big enough to facilitate coexistence, but some subset of rule-breaking fishers decided to violate the boundary and fish inside the exclusion zone (for real-world examples, see~\cite{iacarella, arias}). In addition, an abrupt population collapse could also occur in the related scenario in which the boundary of the exclusion zone remained fixed, but some exogenous factor (e.g., a feature of global change such as water temperature) reduced the size of the habitable domain such that the MPA is no longer sufficiently large~\cite{bruno, soto}.

We established the existence of nontrivial coexistence equilibria in one and higher spatial dimensions without relying on local bifurcation from semi-trivial steady states. Our proofs, based on a topological degree theory argument, demonstrate that coexistence is possible whenever the predator-free exclusion zone is sufficiently large. This requirement can be interpreted as a geometric threshold: when the prey have enough refuge from predation, they can maintain their density above their Allee threshold and thereby support a persistent predator population. The analysis in higher dimensions further reveals that prey densities far from the predation zone approach their carrying capacity (i.e., the prey density is relatively unaffected by the distant predator population). Our results extend several earlier lines of research on exclusion zones (often referred to in the literature as protection zones). Prior studies, beginning with Du and Shi in~\cite{DuShi2006} and further developed in (to name a few)~\cite{CuiShiWu2014,MinWang2018,ZhangLouWang2022,Yang2023}, primarily use bifurcation theory to establish results on the relationships between  coexistence equilibria and semi-trivial equilibria. By contrast, our approach provides a global existence framework, establishing coexistence far from bifurcation points and for broad parameter ranges.

We also developed key results for the case of \emph{thin-limits} in which the domain size $L$ is fixed and the size $a$ of the predation zone approaches zero as a limit. Surprisingly, we found that the predator does not always go extinct even as the region it can occupy shrinks toward zero. This result, which is more interesting and relevant for versions of the model with multiple spatial dimensions, has the predator persisting in a narrow band or face of the domain by virtue of the influx of prey across the boundary from the (much larger) exclusion zone. This thin-limit result is independent of the existence of an Allee effect for the prey. From an ecological perspective, this result is reminiscent of the biodiversity-enhancing effects of \emph{spatial subsidies} in which resources from an inaccessible area (e.g., an offshore marine area) arrive by virtue of the tides, wind, or even animal agents, providing biomass or nutrient inputs to adjacent regions. Such spatial subsidies can be especially important when the difference between source and destination regions is pronounced (e.g., desert land surrounded by productive seas~\cite{polis, stapp}) or when the magnitude of the cross-boundary flow is large (e.g., salmon swimming upstream to spawn and die~\cite{wipfli}).

In addition to the above analytical results, we also explored the model numerically, focusing on the bifurcations that occur as the size of the exclusion zone changes. We observed that, depending on parameter values, we may see solution approaching coexistence equilibria, extinction, limit cycles, and solutions displaying chaotic behavior. We observe that the size of exclusion zone which maximizes the long-term predator may be very large, or very small, or anywhere in between. Perhaps most surprising is the case when the maximum predator population is achieved for large exclusion zone (when the predation zone is vanishingly small), which is observed numerically and confirmed analytically in the thin limit analysis. For a related optimality problem with MPAs, see Neubert (2003)~\cite{neubert}.

Outside of the context of MPAs, we also believe that our results (particularly the results above pertaining to the existence of alternative regimes with and without an optimal size for the predator exclusion zone) are relevant for ecological scenarios in which predator species concentrate (or are constrained so that they concentrate) their activities in particular subregions of space. One example stems from a study of servals (mid-sized members of the cat family, which feed primarily on rodents) living near a developed industrial area in South Africa~\cite{loock}. In that system, a combination of infrastructure and persecution limit access by large-bodied carnivores but not servals, permitting the servals to reach atypically high densities in the ecological buffer zone around the industrial estate. The phenomenon is interpretable as an example of \emph{mesopredator release}~\cite{prugh, jachowski} in which populations of small-bodied carnivores thrive when populations of large-bodied carnivores are suppressed, thus reducing their roles as competitors and intra-guild predators of the small-bodied species. A second relevant example stems from the recent discovery that species of mammalian carnivores differ greatly in the degree to which they concentrate their movement along \emph{travel routeways} within their home ranges~\cite{catdog}. In particular, canids (members of the dog family) possess, on average, home range distributions that feature substantially greater densities of travel routeways than do felids (members of the cat family), though there are counter-examples within each family of carnivores. Importantly, canids also spend a greater fraction of their time moving along their routeways than do felids.  The analyses in that work are based on so-called \emph{probability ridges} and ridge-associated probability mass~\cite{catdog}, with the upshot being that high levels of ridge-associated probability mass reflect a form of spatial concentration of activity by a predator that is similar in some ways to the consumer exclusion zones that helped motivate our work here. The analogy is admittedly loose in that route-following predators will, of course, sometimes leave those routes to hunt. However, there is a parallel in the realm of resource management wherein several classes of MPAS exist that differ in the porosity (or stringency of enforcement) of the exclusion zone boundaries~\cite{ecosta}.  

Overall, we have found that a diffusive predator--prey model with an exclusion zone limiting the predator and an Allee effect for the prey is capable of a rich array of dynamics. Beyond the ``usual suspects'' of two species coexistence and Hopf bifurcations to periodic predator--prey cycles, we also found that the model can exhibit a thin-limit with a positive predator population size when the exclusion zone is very large and sudden, catastrophic collapse from coexistence to mutual extinction if the exclusion zone becomes too small. Lastly, we found it especially interesting that the complex outcomes for this model were achieved largely through ``spatial changes'' traceable to alteration of the size of the predation zone and the predator-free exclusion zone rather than to the addition of ever more complex expressions for population growth, functional responses, and modes of dispersal. Some of this complexity was, of course, only possible through ``spatial changes'' in combination with the nonlinearity induced by the Allee effect, but the fact remains that assumptions about spatial structure can matter quite strongly, even in a model as well studied as the Lotka-Volterra predator--prey system.  


\subsection{Future work}

Many possibilities exist for extending this work. First, the question of the optimal predator domain remains open: given a fixed prey domain $\Omega$, what choice of predator domain $A\subset\Omega$ maximizes the long-term predator population or the total biomass of both species? A related question is whether multiple disjoint predator or exclusion zones can support more predators than a single contiguous zone. Future work will addressing these problems (either analytically or numerically) to explain the interactions between geometry and ecology.

Second, our thin-limit analysis could be extended to higher dimensions. Deriving and studying the limiting equations when the predator domain collapses to a codimension-1 manifold (such as a coastline or predator corridor) would connect PDE analysis with realistic ecological scenarios. Numerical investigation of these singularly coupled systems will further reveal how ecological phenomena respond to geometric properties such as curvature of the predation zone.

Third, our computations indicate that the system may support a wide variety of bifurcations, including Hopf and saddle-node bifurcations of equilibria and limit cycles, and possibly a period-doubling route to chaos. An analytical treatment these bifurcations would substantially deepen our understanding of the complexity introduced to a predator--prey system by the presence of an exclusion zone.

In summary, this study establishes conditions for coexistence in predator--prey systems with exclusion zones and the Allee effect, showing how spatial structure can sustain or disrupt the survival of both species. Our analytical and numerical results together reveal rich behaviors shaped by domain geometry and diffusion, and lay the groundwork for future studies.

\section*{Declarations}
	
	\noindent\textbf{Funding.} W. F. Fagan was supported by NSF DMS-2451241 The Mathematics of Animal Migration.

	\medskip
	
	\noindent\textbf{Conflict of Interest.} The authors declare that they have no conflicts of interest or competing interests associated with this work.
	
	\medskip
	
	\noindent\textbf{Data Availability.} No external data were used in this manuscript. Data generated in this manuscript (in particular, the data shown in the figures in Section~\ref{sec:numerics}) are available upon request. Additionally, the software used in Section~\ref{sec:numerics} to solve the one dimensional model is available at \url{https://github.com/alexsafsten/ExclusionZone1D}.

\appendix

\section{Proof of Proposition \ref{prop:unique_limit}}\label{app:unique_limit_proof}

\begin{proof}
    By Proposition~\ref{prop:subsolzet} we may select $b_1$ so that if $L>b_1$, then $\zeta_{0,L}(L)>\theta'$. We will show that there exists $b_0\geq b_1$ such that~\eqref{eq:unique_limit} has a unique solution $u$ whenever $L>L_0$ such that $\zeta_{0,L}<u<1$. Then letting $\eta_0=1-\zeta_{0,b_0}(b_0)$, the existence and uniqueness aspect of the result is proved. The fact that this solution is also monotone and $u(L)$ is increasing with $u(L)\to 1$ as $L\to \infty$ will follow from the proof of uniqueness.

    Observe that $\zeta_{0,L}$ is a subsolution of~\eqref{eq:unique_limit} and the constant $1$ is a supersolution. Therefore, there exists a solution $u$ with $\zeta_{0,L}<u<1$. It remains to show that, if $L$ is sufficiently large, this is the only solution such that $\zeta_{a_0,L}<u<1$. We begin by showing that such a solution is monotone increasing. Multiplying~\eqref{eq:unique_limit} by $u'$ and integrating, we find that $u$ satisfies the conservation law~\eqref{eq:energy}, so for any $x$,
    \begin{equation}\label{eq:u_energy}
        \frac{d_u}{2}(u'(x))^2+F(u(x))=F(u(L)).
    \end{equation}
    Since $u(L)>\zeta(L)>\theta'$, $F(u(L))>0$. Whenever $F$ is positive, it is one-to-one, so we conclude that if $u'(x)=0$ for any $x$, then $u(x)=u(L)$. Therefore, $u$ cannot achieve an internal maximum at any point $x_0\in (0,L)$ because at this maximum, we would have $u(x_0)=u(L)$, but then either $u$ is constant on $[x_0,L]$ or $u$ achieves a minimum at a point $x_1\in (x_0,L)$ where $u(x_1)<u(L)$. Neither of these is possible. Similarly, $u$ cannot achieve a local minimum. We conclude that $u$ is monotone. Since $u(L)>\theta$, $u'(L)=0$, and $-d_u u''=f(u)$, we have $u''(x)<0$ for $x$ near $L$ and thus, $u'(x)>0$ for such $x$. But since $u$ is monotone, we must therefore have $u'(x)>0$ for all $x\in [0,L)$.

    Since $u$ is increasing, we can use~\eqref{eq:u_energy} to write $u'$ as:
    \begin{equation}
        u'=\sqrt{\frac{2}{d_u}\left(F(u(L))-F(u)\right)}
    \end{equation}
    Again, due to the strict monotonicity of $u$, we may invert this differential equation to compute $x$ as a function of $u$. Then, letting $q=u(L)$ and $p=\gamma/\alpha$, we may write $L$ as a function of $q$:
    \begin{equation}\label{eq:L_map}
        L=\mathcal L(q)=\sqrt{\frac{d_u}{2}}\int_p^q\frac{1}{\sqrt{F(q)-F(u)}}\,du.
    \end{equation}
    Thus,~\eqref{eq:unique_limit} has a monotone increasing solution $u$ on $[0,L]$ such that $u(L)=q<1$ precisely when $L=\mathcal L(q)$. If there were two such solutions  $u_1$ and $u_2$, we would have $u_1(L)=q_1$ and $u_2(L)=q_2$. If $q_1=q_2$, then by uniqueness of solutions to initial value problems, $u_1=u_2$. Otherwise, we have $\mathcal L(q_1)=\mathcal L(q_2)$. We will show that for $q$ sufficiently close to $1$, $\mathcal L(q)$ is strictly increasing, establishing uniqueness.

    Let $\nu=-f'(1)>0$ and let $\delta>0$ be fixed such that both of the following hold for $1-2\delta<s<1$:
    \begin{itemize}
        \item $-2\nu\leq f'(s)\leq -\nu/2$
        \item $\frac{\nu}{2}(1-s)<f(s)<2\nu(1-s)$
    \end{itemize}
    We also assume $1-2\delta>\theta'$. Suppose $q>1-\delta$. Then we break up~\eqref{eq:L_map} into two integrals:
    \begin{align*}
        \mathcal L(q)&=A(q)+B(q)\\
        &=\sqrt{\frac{d_u}{2}}\int_p^{q-\delta}\frac{1}{\sqrt{F(q)-F(u)}}\,du+\sqrt{\frac{d_u}{2}}\int_{q-\delta}^q\frac{1}{\sqrt{F(q)-F(u)}}\,du.
    \end{align*}

    Since $q>\theta'$, there is no $u\leq q-\delta$ such that $F(u)=F(q)$. Then the function $(u,q)\mapsto F(q)-F(u)$ is continuous and positive on the compact set
    \begin{equation*}
        T=\{(u,q)\in\mathbb R^2:1-\delta\leq q\leq 1,\;p\leq u\leq q-\delta\},
    \end{equation*}
    so we conclude that there exists $m>0$ so that $F(q)-F(u)\geq m$ for all $1-\delta\leq q\leq 1$ and $p\leq u\leq q-\delta$. Thus,
    \begin{equation*}
        |A(q)|\leq\sqrt{\frac{d_u}{2}} \frac{(1-p)}{\sqrt{m}}.
    \end{equation*}
    and
    \begin{align*}
        |A'(q)|&=\sqrt{\frac{d_u}{2}}\left|\frac{1}{\sqrt{F(q)-F(q-\delta)}}-\frac{F'(q)}{2}\int_p^{q-\delta}\frac{1}{(F(q)-F(u))^{3/2}}\,du\right|\\
        &\leq\sqrt{\frac{d_u}{2}}\left(\frac{1}{\sqrt{m}}+\frac{\Vert f\Vert_{C^0}}{2}\frac{(1-p)}{m^{3/2}}\right).
    \end{align*}
    Therefore, both $A(q)$ and $A'(q)$ are uniformly bounded as $q\to 1^-$.  

    Now we show $\lim_{q\to 1^-}B(q)=\lim_{q\to 1^-}B'(q)=\infty$ so for $q$ sufficiently close to $1$, $\mathcal L(q)$ is increasing and unbounded. We may write
    \begin{equation*}
        B(q)=\sqrt{\frac{d_u}{2}}\int_0^\delta\frac{1}{\sqrt{F(q)-F(q-s)}}\,ds.
    \end{equation*}
    That $B(q)$ is differentiable is nontrivial because the integrand is unbounded. Therefore, we first show that $B(q)$ is differentiable and that the expression for the derivative can be found by differentiating under the integral sign, as expected. 

    If $B'(q)$ exists, it is the limit of the following finite difference as $h\to 0$:
    \begin{equation}\label{eq:finite_difference}
        \frac{B(q+h)-B(q)}{h}=\sqrt{\frac{d_u}{2}}\int_0^\delta\frac{1}{h}\left(\frac{1}{\sqrt{F(q+h)-F(q+h-s)}}-\frac{1}{\sqrt{F(q)-F(q-s)}}\right)\,ds.
    \end{equation}
    We will show that there exists a constant $C$ such that for $|h|$ sufficiently small, the integrand of~\eqref{eq:finite_difference} is bounded by $C/\sqrt{s}$, which is integrable. Therefore, by the Lebesgue Dominated Convergence theorem, we may pass the limit into the integral.
    
    First, we note that 
    \begin{equation}
        F(q)-F(q-s)=\int_{q-s}^q f(s')\,ds',
    \end{equation}
    so we may bound this quantity for $1-2\delta<q-s<q<1$:
    \begin{equation}\label{eq:F_bound}
        \frac{\nu}{4}s (s+2\varepsilon)=\frac{\nu}{2}\int_{q-s}^q (1-s')\,ds'\leq F(q)-F(q-s)\leq 2\nu\int_{q-s}^q (1-s')\,ds'= \nu s (s+2\varepsilon),
    \end{equation}
    where $\varepsilon=1-q$. Thus, the integrand in~\eqref{eq:finite_difference} is bounded as follows, assuming $|h|<\varepsilon/2$:
    \begin{align*}
        &\left|\frac{1}{h}\left(\frac{1}{\sqrt{F(q+h)-F(q+h-s)}}-\frac{1}{\sqrt{F(q)-F(q-s)}}\right)\right|\\
        &=\frac{|\sqrt{F(q)-F(q-s)}-\sqrt{F(q+h)-F(q+h-s)}|}{h\sqrt{F(q)-F(q-s)}\sqrt{F(q+h)-F(q+h-s)|}}\\
        &=\frac{8|F(q)-F(q-s)-F(q+h)-F(q+h-s)|}{h\left((F(q)-F(q-s))\sqrt{F(q+h)-F(q+h-s)}+(F(q+h)-F(q+h-s))\sqrt{F(q)-F(q-s)}\right)}\\
        &\leq \frac{8\left|\int_{q-s}^q\int_{y+h}^y f'(z)\,dz\,dy\right|}{h(\nu s)^{3/2}\left((s+2\varepsilon)\sqrt{s+2(\varepsilon-h)}+(s+2(\varepsilon-h))\sqrt{s+2\varepsilon}\right)}\\
        &\leq\frac{16\nu hs}{h(\nu s\varepsilon)^{3/2}\left(2+\sqrt{2}\right)}\\
        &\leq \frac{C}{\sqrt{s}},\quad \text{where}\quad C=\frac{16}{\sqrt{\nu\varepsilon^3}(2+\sqrt{2})}.
    \end{align*}
    Thus, $B(q)$ is differentiable and    
    \begin{equation*}
        B'(q)=\frac{1}{2}\sqrt{\frac{d_u}{2}}\int_0^\delta\frac{F'(q-s)-F'(q)}{(F(q)-F(q-s))^{3/2}}\,ds.
    \end{equation*}
    We may formulate a similar bound to~\eqref{eq:F_bound} for $F'$:
    \begin{equation*}
        F'(q-s)-F'(q)=-\int_{q-s}^{q}f'(r)\,dr\geq\frac{\nu s}{2}.
    \end{equation*}
    Therefore, we calculate
    \begin{equation}\label{eq:B_blowup}
    \begin{split}
        B(q)&\geq \sqrt{\frac{d_u}{2\nu}}\int_0^\delta\frac{1}{\sqrt{s(s+2\varepsilon)}}\,ds=\left.-\sqrt{\frac{2d_u}{\nu}}\log \left(\sqrt{s+2 \varepsilon }-\sqrt{s}\right)\right|_0^\delta=\sqrt{\frac{2d_u}{\nu}}\arcsinh\left(\sqrt{\frac{\delta}{2\varepsilon}}\right),
    \end{split}
    \end{equation}
    and furthermore,
    \begin{align}\label{eq:B_prime_blowup}
        B'(q)\geq\frac{1}{4}\sqrt{\frac{d_u}{2\nu}}\int_0^\delta\frac{1}{\sqrt{s}(s+2\varepsilon)^{3/2}}\,ds=\left.\frac{1}{4}\sqrt{\frac{d_u}{2\nu}}\frac{\sqrt{s}}{\varepsilon  \sqrt{s+2 \varepsilon }}\right|_0^\delta=\frac{1}{4\varepsilon}\sqrt{\frac{d_u}{2\nu}}\sqrt{\frac{\delta }{\delta +2 \varepsilon }}.
    \end{align}
    The expressions~\eqref{eq:B_blowup} and~\eqref{eq:B_prime_blowup} show that $B(q)$ and $B'(q)$ both approach $+\infty$ as $q\to 1^-$ (which corresponds to $\varepsilon\to 0^+$). Since $A(q)$ and $A'(q)$ are both bounded, we conclude that there exists $q_0$ with $1-\delta<q_0<1$ so that $L(q)$ is increasing for $q_0<q<1$ and moreover, $L(q)>L(q')$ for each $q'<q_0$. Let $b_0=\max\{\mathcal L(q_0),b_1\}$ and $\eta_0=\zeta_{0,b_0}$. Then for all $L>b_0$,~\eqref{eq:unique_limit} has a unique solution $u$ so that $1-\eta_0<u(L)<1$. Moreover, we have already shown that this solution is monotone increasing, and that as $L\to\infty$, $u(L)$ must monotonically approach $1$.
\end{proof}

\printbibliography
    
\end{document}